\documentclass[smallextended]{svjour3}

\smartqed  

\usepackage{appendix}
\usepackage{amssymb}
\usepackage{empheq}
\usepackage[mathscr]{euscript}
\usepackage{stmaryrd}
\usepackage{mathabx}
\usepackage{dsfont}
\usepackage{graphicx}
\usepackage{multirow}
\usepackage{algorithm2e}
\usepackage[table,dvipsnames]{xcolor}
\usepackage{color}
\usepackage{arydshln}
\usepackage{tikz}
\usetikzlibrary{3d}
\usetikzlibrary{decorations.pathreplacing,angles,quotes,patterns}
\usepackage{pgfplots}
\usepgfplotslibrary{colorbrewer}
\pgfplotsset{cycle list/Set1-4}
\usepackage{algorithm2e}
\usepackage{braket}
\usepackage[bb=boondox]{mathalfa}
\usepackage{subcaption}
\usepackage{booktabs}
\usepackage{xfrac}

\usepackage{tikz}
\usetikzlibrary{arrows}

\newcommand{\blue}[1]{{#1}}

\usepackage[colorlinks=true,breaklinks=true,bookmarks=true,urlcolor=blue,citecolor=blue,linkcolor=blue,bookmarksopen=false,draft=false]{hyperref}

\newcommand{\bbN}{\mathbb{N}}
\newcommand{\bbR}{\mathbb{R}}
\newcommand{\bbZ}{\mathbb{Z}}

\newcommand{\bbone}{\mathbb{1}}
\newcommand{\Conv}{\operatorname{Conv}}
\newcommand{\ext}{\operatorname{ext}}
\newcommand{\gr}{\operatorname{gr}}
\newcommand{\Proj}{\operatorname{Proj}}

\newcommand\defeq{\mathrel{\overset{\makebox[0pt]{\mbox{\normalfont\tiny\sffamily def}}}{=}}}

\newcommand{\NL}{\texttt{NL}}
\newcommand{\NN}{\texttt{NN}}

\newcommand{\ReLu}{\texttt{ReLU}}
\newcommand{\Leaky}{\texttt{Leaky}}

\newcommand{\Max}{\texttt{Max}}

\newcommand{\trans}{\mathsf{Transport}}
\newcommand{\tw}{\widetilde{w}}
\newcommand{\tx}{\widetilde{x}}

\newcommand{\Rcayley}{R_{\operatorname{cayley}}}
\newcommand{\Rextended}{R_{\operatorname{extended}}}
\newcommand{\Rideal}{R_{\operatorname{ideal}}}
\newcommand{\Rsharp}{R_{\operatorname{sharp}}}

\newcommand{\Smax}{S_{\operatorname{max}}}
\newcommand{\Scayley}{S_{\operatorname{cayley}}}

\newcommand{\ubar}[1]{\underline{#1}}
\renewcommand{\bar}[1]{\overline{#1}}

\usepackage{pifont}

\newtheorem{assumption}{Assumption}
\newtheorem{observation}{Observation}

\usepackage{mdframed}

\mdfdefinestyle{theoremstyle}{%
frametitlerule=true,%
frametitlebackgroundcolor=gray!20,
innertopmargin=\topskip,
}
\mdtheorem[style=theoremstyle]{theo}{Main result for ReLU networks (informal)}

\begin{document}

\title{Strong mixed-integer programming formulations for trained neural networks\footnote{\blue{An extended abstract version of this paper appeared in \cite{Anderson:2019}.}}}

\titlerunning{MIP formulations for trained neural networks}

\author{Ross Anderson \and Joey Huchette \and Will Ma \and Christian Tjandraatmadja \and Juan Pablo Vielma}

\institute{R. Anderson \at \email{rander@google.com} \and J. Huchette \at \email{joehuchette@rice.edu} \and W. Ma \at \email{wm2428@gsb.columbia.edu} \and C. Tjandraatmadja \at \email{ctjandra@google.com} \and J. P. Vielma \at \email{jvielma@mit.edu, jvielma@google.com}}

\maketitle

\begin{abstract}
We present strong mixed-integer programming (MIP) formulations for high-dimensional piecewise linear functions that correspond to trained neural networks. These formulations can be used for a number of important tasks, such as verifying that an image classification network is robust to adversarial inputs, or solving decision problems where the objective function is a machine learning model. We present a generic framework, which may be of independent interest, that provides a way to construct sharp or ideal formulations for the maximum of $d$ affine functions over arbitrary polyhedral input domains. We apply this result to derive MIP formulations for a number of the most popular nonlinear operations (e.g. ReLU and max pooling) that are strictly stronger than other approaches from the literature. We corroborate this computationally, showing that our formulations are able to offer substantial improvements in solve time on verification tasks for image classification networks.
\end{abstract}



\begin{acknowledgements}
The authors gratefully acknowledge Yeesian Ng and Ond\v{r}ej S\'{y}kora for many discussions on the topic of this paper, and for their work on the development of the \texttt{tf.opt} package used in the computational experiments.
\end{acknowledgements}

\section{Introduction}

Deep learning has proven immensely powerful at solving a number of important predictive tasks arising in areas such as image classification, speech recognition, machine translation, and robotics and control~\cite{Goodfellow:2016,LeCun:2015}. The workhorse model in deep learning is the feedforward network \blue{$\NN{}:\bbR^{m_0} \to \bbR^{m_s}$ that maps a (possibly very high-dimensional) input $x^0 \in \bbR^{m_0}$ to an output $x^s = \NN{}(x^0)\in \bbR^{m_s}$. A feedforward network with $s$ layers can be recursively described as }
\begin{equation} \label{eqn:simple-feedforward-network}
    x^i_j = \texttt{NL}^{i,j}(w^{i,j} \cdot x^{i-1} + b^{i,j}) \quad \blue{ \forall i \in \llbracket s \rrbracket,\; j \in \llbracket m_i \rrbracket}
\end{equation}
\blue{where $\llbracket n \rrbracket \defeq \{1,\ldots,n\}$, $m_i$ is both the number of \emph{neurons} in layer $i$ and the output dimension of the neurons in layer $i-1$ (with the input $x^0 \in \bbR^{m_0}$ considered to be the $0$-th layer). Furthermore, for each $i \in \llbracket s \rrbracket$ and $j \in \llbracket m_i \rrbracket$,  $\texttt{NL}^{i,j}:\bbR \to \bbR$ is some simple nonlinear \emph{activation function} that is fixed before training}, and $w^{i,j}$ and $b^{i,j}$ are the \emph{weights} and \emph{bias} of an affine function which is learned during the training procedure. In its simplest and most common form, the \blue{activation function} would be the rectified linear unit (ReLU), defined as $\ReLu{}(v) = \max\{0,v\}$. 

Many standard nonlinearities $\NL{}$, such as the ReLU, are \emph{piecewise linear}: that is, there exists a partition $\{S^i \subseteq D\}_{i=1}^d$ of the domain and affine functions $\{f^i\}_{i=1}^d$ such that $\NL{}(x) = f^i(x)$ for all $x \in S^i$.
If all nonlinearities describing $\NN{}$ are piecewise linear, then the entire network $\NN{}$ is piecewise linear as well, \blue{and conversely, any continuous piecewise linear function can be represented by a ReLU neural network~\cite{arora2016understanding}}. \blue{Beyond the ReLU, we also investigate activation functions that can be expressed as the maximum of finitely many affine functions, a class which includes certain common operations such as max pooling and reduce max.}

There are numerous contexts in which one may want to solve an optimization problem containing a trained neural network such as $\NN{}$. For example, this paradigm arises in deep reinforcement learning problems with high dimensional action spaces and where any of the cost-to-go function, immediate cost, or the state transition functions are learned by a neural network~\cite{Arulkumaran:2017,Dulac-Arnold:2015,Mladenov:2017,Ryu:2019,Say:2017,Wu:2017}. \blue{More generally, this approach may be used to approximately optimize learned functions that are difficult to model explicitly~\cite{Grimstad:2019,Lombardi:2018,Lombardi:2017empirical}.} Alternatively, there has been significant recent interest in verifying the robustness of trained neural networks deployed in systems like self-driving cars that are incredibly sensitive to unexpected behavior from the machine learning model~\cite{Carlini:2016,Papernot:2016,Szegedy:2013}. Relatedly, a string of recent work has used optimization over neural networks trained for visual perception tasks to \emph{generate} new images which are ``most representative'' for a given class~\cite{Olah:2017}, are ``dreamlike''~\cite{Mordvintsev:2015}, or adhere to a particular artistic style via neural style transfer~\cite{Gatys:2015}.

\subsection{MIP formulation preliminaries} \label{sec:preliminaries}

In this work, we study mixed-integer programming (MIP) approaches for optimization problems containing trained neural networks. In contrast to heuristic or local search methods often deployed for the applications mentioned above, MIP offers a framework for producing provably optimal solutions. This is particularly important, for example, in verification problems, where rigorous dual bounds can guarantee robustness in a way that purely primal methods cannot.

\blue{Let $f:D\subseteq \bbR^\eta \to \bbR$ be a $\eta$-variate affine function with domain $D$, $\NL{}:\bbR\to\bbR$ be an univariate nonlinear activation function, and $\circ$ be the standard function composition operator so that  $(\NL{} \circ f)(x) = \NL{}(f(x))$. Functions of the form $(\NL{} \circ f)(x)$ precisely describe an individual neuron, with $\eta$ inputs and one output. A standard way to model a function $g:D\subseteq \bbR^\eta\to \bbR$ using MIP is to construct a formulation of its graph defined by $\gr(g; D) \defeq \Set{ \left(x,y\right)\in \bbR^\eta\times \bbR | x \in D, \: y=g(x) }$.}
We \blue{therefore} focus on constructing MIP formulations for the \emph{graph} of individual neurons:
\begin{equation} \label{eqn:single-neuron}
\gr(\NL{} \circ f; D) \defeq \Set{ \left(x,\blue{y}\right)\blue{\in \bbR^\eta\times \bbR } | x \in D,\:  \blue{y=(\NL{} \circ f) (x)} },
\end{equation}
 because\blue{, as we detail at the end of this section,} we can readily produce a MIP formulation for the entire network as the composition of formulations for each individual neuron. 
\begin{definition} \label{def:ideal}
    Throughout, we will notationally use the convention that $x \in \bbR^\eta$ \blue{and $y \in \bbR$ are respectively the \emph{input} and  \emph{output} variables of the function $g:D\subseteq \bbR^\eta\to \bbR$ we are modeling. In addition, $v \in \bbR^h$ and $z \in \bbR^q$ are respectively any auxiliary continuous and binary variables used to construct a formulation of $\gr(g; D)$}. The orthogonal projection operator $\Proj$ will be subscripted by the variables to project onto, e.g. $\Proj_{x,y}(R) = \Set{(x,y) | \exists v, z \text{ s.t. } (x,y,v,z) \in R}$ is the orthogonal projection of $R$ onto the ``original space'' of $x$ and $y$ variables. We denote by $\text{ext}(Q)$ the set of extreme points of a polyhedron $Q$.
    
    Take  $S\defeq\gr(g; D) \subseteq \bbR^{\eta+1}$ \blue{to be the set} we want to model, and a polyhedron $Q \subset \bbR^{\eta + 1 + \blue{h} + q}$. Then:
    \begin{itemize}
        \item A \emph{(valid) mixed-integer programming (MIP) formulation} of $S$ consists of the linear constraints on $(x,y,v,z)\in\bbR^{\eta + 1 + \blue{h} + q}$ which define a polyhedron $Q$, combined with the integrality constraints $z \in \{0,1\}^q$, such that
        \[
            S = \Proj_{x,y}\left(Q \cap (\bbR^{\eta + 1 + \blue{h}} \times \{0,1\}^q)\right).
        \]
        We refer to $Q$ as the \emph{linear programming (LP) relaxation} of the formulation.
        \item A MIP formulation is \emph{sharp} if $\Proj_{x,y}(Q) = \Conv(S)$.
        \item A MIP formulation is \emph{hereditarily sharp} if fixing any subset of binary variables $z$ to 0 or 1 preserves sharpness.
        \item A MIP formulation is \emph{ideal} (or \emph{perfect}) if $\ext(Q) \subseteq \bbR^{\eta + 1 + \blue{h}} \times \{0,1\}^q$.
        \item The \emph{separation} problem for a family of inequalities is to find a valid inequality violated by a given point or certify that no such inequality exists.
        \item An inequality is \emph{valid} for the formulation if each integer feasible point in $\ubar{Q} = \Set{(x,y,v,z) \in Q | z \in \{0,1\}^q}$ satisfies the inequality. Moreover, a valid inequality is \emph{facet-defining} if the dimension of the points in $\ubar{Q}$ satisfying the inequality at equality is exactly one less than the dimension of $\ubar{Q}$ itself.
    \end{itemize}
\end{definition}

Note that ideal formulations are sharp, but the converse is not necessarily the case~\cite[Proposition 2.4]{Vielma:2015}.  In this sense, ideal formulations offer the tightest possible relaxation, and the integrality property in Definition~\ref{def:ideal} tends to lead to superior computational performance. Furthermore, note that a hereditarily sharp formulation is a formulation which retains its sharpness at every node in a branch-and-bound tree, and as such is potentially superior to a formulation which only guarantees sharpness at the root node~\cite{Jeroslow:1984,Jeroslow:1988c}. Additionally, it is important to keep in mind that modern MIP solvers will typically require an explicit, finite list of inequalities defining $Q$.

\blue{Finally, for each  $i \in \llbracket s \rrbracket$ and $j \in \llbracket m_i \rrbracket$ let  
$Q_{i,j}\subseteq\bbR^{m_{i-1} + 1 + h_{i,j} + q_{i,j}}$ be a polyhedron such that $\gr(\NL^{i,j}{} \circ f^{i,j}; D_{i,j}) = \Proj_{x,y}\left(Q_{i,j} \cap (\bbR^{m_{i-1} + 1 + h_{i,j}} \times \{0,1\}^{ q_{i,j}})\right)$, where each $D_{i,j}$ is the domain of neuron $\NL^{i,j}$ and $f^{i,j}(x) = w^{i,j} \cdot x + b^{i,j}$ is the learned affine function in network \eqref{eqn:simple-feedforward-network}. Then, a formulation for the complete network \eqref{eqn:simple-feedforward-network} is given by 
\[ \left(x^{i-1},x^i_j,v^{i,j},z^{i,j}\right) \in \left(Q_{i,j} \cap (\bbR^{m_{i-1} + 1 + h_{i,j}} \times \{0,1\}^{ q_{i,j}})\right)\quad \forall i \in \llbracket s \rrbracket,\; j \in \llbracket m_i \rrbracket.\]
Any properties of the individual-neuron formulations $Q_{i,j}$ (e.g. being ideal), are not necessarily preserved for the combined formulation for the complete network. However, for similarly sized-formulations, combining stronger formulations (e.g. ideal instead of sharp) usually leads to complete formulations that are more computationally effective \cite{Vielma:2015}. Hence, our focus for all theoretical properties of the formulations will be restricted to individual-neuron formulations. }



\subsection{Our contributions}

We highlight the contributions of this work as follows.
\begin{enumerate}
    \item \emph{Generic recipes for building strong MIP formulations of the maximum of $d$ affine functions for any bounded polyhedral input domain.}
    \begin{itemize}
        \item \textbf{[Propositions~\ref{prop:max_ideal_set_primal} and \ref{prop:max_ideal_set_dual}]} We derive both primal and dual characterizations for ideal formulations via the \emph{Cayley embedding}.
        \item \textbf{[Propositions~\ref{prop:max_sharp_set_primal} and \ref{prop:max_sharp_set_dual}]} We relax the Cayley embedding in a particular way, and use this to derive simpler primal and dual characterizations for (hereditarily) sharp formulations.
        \item We discuss how to separate both dual characterizations via subgradient descent.
    \end{itemize}
    \item \emph{Simplifications for common special cases.}
    \begin{itemize}
        \item \textbf{[Corollary~\ref{cor:max_2_ideal_general}]} We show the equivalence of the ideal and sharp characterizations when $d=2$ \blue{(i.e. the maximum of two affine functions)}.
        \item \textbf{[Proposition~\ref{prop:transportationDuality}]} We show that, if the input domain is a product of simplices, the separation problem of the sharp formulation can be reframed as a series of transportation problems.
        \item \textbf{[Corollaries~\ref{cor:simplicesD2Result} and \ref{cor:simplicesP2}, and Proposition~\ref{prop:simplicesD2FacetDefining}]} When the input domain is a product of simplices, and either (1) $d=2$, or (2) each simplex is two-dimensional, we provide an explicit, finite description for the sharp formulation. Furthermore, none of these inequalities are redundant.
    \end{itemize}
    \item \emph{Application of these results to construct MIP formulations for neural network nonlinearities.}
    \begin{itemize}
        \item \textbf{[Propositions~\ref{prop:ideal-relu} and \ref{prop:relu-separation}]} We derive an explicit ideal formulation for the ReLU nonlinearity over a box input domain, the most common case. Separation over this ideal formulation can be performed in time linear in the input dimension.
        \item \textbf{[Corollary~\ref{cor:simplicesD2ReLU}]} We derive an explicit ideal formulation for the ReLU nonlinearity where some (or all) of the input domain is \emph{one-hot encoded} categorical or discrete data. Again, the separation can be performed efficiently, and none of the inequalities are redundant.
        \item \textbf{[Propositions~\ref{prop:max_sharp} and~\ref{prop:max_sharp_sep}]} We present an explicit hereditarily sharp formulation for the maximum of $d$ affine functions over a box input domain, and provide an efficient separation routine. Moreover, a subset of the defining constraints serve as a tightened big-$M$ formulation.
        \item \textbf{[Proposition~\ref{prop:ideal-leaky-relu}]} We produce similar results for more exotic neurons: the leaky ReLU and the clipped ReLU (see the online supplement).
    \end{itemize}
    \item \emph{Computational experiments on verification problems arising from image classification networks trained on the MNIST digit data set.}
    \begin{itemize}
        \item We observe that our new formulations, along with just a few rounds of separation over our families of cutting planes, can improve the solve time of Gurobi on verification problems by orders of magnitude.
    \end{itemize}
\end{enumerate}
Our contributions are depicted in Figure~\ref{fig:results}.  It serves as a roadmap of the paper.

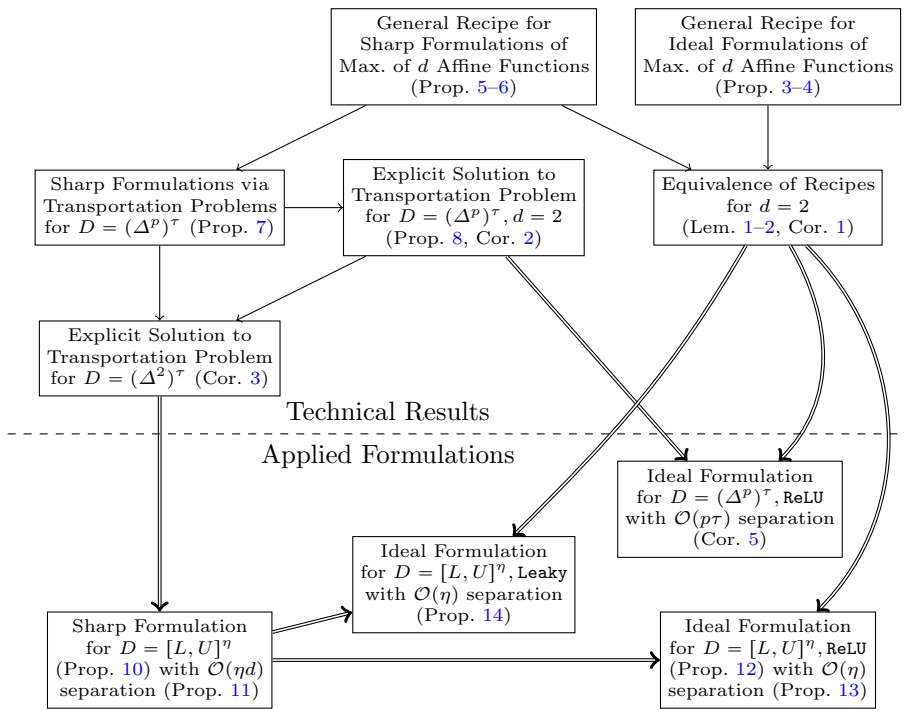
\begin{figure}
\centering
\begin{tikzpicture}[font=\scriptsize]
\node[draw, align=center] (A) at (8,2) {General Recipe for \\ Ideal Formulations of \\ Max.~of $d$ Affine Functions\\ (Prop.~\ref{prop:max_ideal_set_primal}--\ref{prop:max_ideal_set_dual})};
\node[draw, align=center] (B) at (4,2) {General Recipe for \\ Sharp Formulations of \\ Max.~of $d$ Affine Functions\\ (Prop.~\ref{prop:max_sharp_set_primal}--\ref{prop:max_sharp_set_dual})};
\node[draw, align=center] (C) at (8,0) {Equivalence of Recipes \\ for $d=2$ \\ (Lem.~\ref{lem:max_2_ideal_lb}--\ref{lem:max_2_ideal_ub}, Cor.~\ref{cor:max_2_ideal_general})};
\draw[->](A)--(C);
\draw[->](B)--(C);
\node[draw, align=center] (D) at (0,0) {Sharp Formulations via \\ Transportation Problems \\ for $D=(\Delta^p)^{\tau}$ (Prop.~\ref{prop:transportationDuality})};
\draw[->](B)--(D);
\node[draw, align=center] (E) at (4,0) {Explicit Solution to \\ Transportation Problem \\ for $D=(\Delta^p)^{\tau},d=2$ \\ (Prop.~\ref{prop:simplicesD2Explicit}, Cor.~\ref{cor:simplicesD2Result})};
\draw[->](D)--(E);
\node[draw, align=center] (H) at (0,-2) {Explicit Solution to \\ Transportation Problem \\ for $D=(\Delta^2)^{\tau}$ (Cor.~\ref{cor:simplicesP2})};
\draw[->](D)--(H);
\draw[->](E)--(H);
\node[draw, align=center] (F) at (7.5,-4) {Ideal Formulation \\ for $D=(\Delta^p)^{\tau},\ReLu$ \\ with $\mathcal{O}(p\tau)$ separation \\ (Cor.~\ref{cor:simplicesD2ReLU})};
\draw[->, double] (C) to [out=300,in=45] (F);
\draw[->, double](E)--(F);
\node[draw, align=center] (G) at (0,-6) {Sharp Formulation \\ for $D=[L,U]^{\eta}$ \\ (Prop.~\ref{prop:max_sharp}) with $\mathcal{O}(\eta d)$ \\ separation (Prop.~\ref{prop:max_sharp_sep})};
\draw[->, double](H)--(G);
\node[draw, align=center] (I) at (4,-5) {Ideal Formulation \\ for $D=[L,U]^{\eta},\Leaky$ \\ with $\mathcal{O}(\eta)$ separation \\ (Prop.~\ref{prop:ideal-leaky-relu})};
\draw[->, double](C) to [out=240,in=45] (I);
\node[draw, align=center] (J) at (8,-6) {Ideal Formulation \\ for $D=[L,U]^{\eta},\ReLu$ \\ (Prop.~\ref{prop:ideal-relu}) with $\mathcal{O}(\eta)$ \\ separation (Prop.~\ref{prop:relu-separation})};
\draw[->, double](C) to [out=315,in=45] (J);
\draw[->, double](G)--(I);
\draw[->, double](G)--(J);
\node at (3,-2.7) {\normalsize Technical Results};
\node at (3,-3.3) {\normalsize Applied Formulations};
\draw[dashed](-2,-3)--(10,-3);
\end{tikzpicture}
\caption{A roadmap of our results.  The single arrows depict dependencies between our technical results, while the double arrows depict the way in which we use these results to establish our applied formulations in the paper.
Notationally, \blue{$d$ is the number of pieces of the piecewise linear function,} $[L,U]^\eta$ is an $\eta$-dimensional box, and $(\Delta^p)^\tau$ refers to a product of $\tau$ simplices, each with $p$ components.} 
\label{fig:results}
\end{figure}



\subsection{Relevant prior work}\label{sec:lit-review}

In recent years a number of authors have used MIP formulations to model trained neural networks~\cite{Bunel:2017,Cheng:2017,Dutta:2017,Fischetti:2018,Huchette:2018,Kumar:2019,Lomuscio:2017,Say:2017,Serra:2018a,Serra:2018,Tjeng:2017,Wu:2017,Xiao:2018}, mostly applying big-$M$ formulation techniques to ReLU-based networks.
When applied to a single neuron of the form \eqref{eqn:single-neuron}, these big-$M$ formulations will not be ideal or offer an exact convex relaxation; see Example~\ref{ex:relu-big-M} for an illustration. Additionally, a stream of literature in the deep learning community has studied convex relaxations~\cite{Bastani:2016,Dvijotham:2018a,Ehlers:2017,Raghunathan:2018,Salman:2019}, primarily for verification tasks.
Moreover, some authors have investigated how to use convex relaxations within the training procedure in the hopes of producing neural networks with a priori robustness guarantees~\cite{Dvijotham:2018,Wong:2017,Wong:2018}.

\blue{The usage of mathematical programming in deep learning, specifically for training predictive models, has also been investigated in \cite{bienstock2018principled}.}
Beyond \blue{mathematical programming} and convex relaxations, a number of authors have investigated other algorithmic techniques for modeling trained neural networks in optimization problems, drawing primarily from the satisfiability, constraint programming, and global optimization communities~\cite{Bartolini:2011,Bartolini:2012,Katz:2017,Lombardi:2016,Schweidtmann:2018}.
\blue{Throughout this stream of the literature, there is discussion of the performance for specific subclasses of neural network models, including}
 binarized~\cite{Khalil:2018} or input convex~\cite{Amos:2017} neural networks.
\blue{Improving our formulations for specific subclasses of neural networks would constitute interesting future work.}

\blue{Outside of applications in machine learning}, the formulations presented in this paper also have connections with existing structures studied in the MIP and constraint programming communities, like indicator variables, on/off constraints, and convex envelopes~\cite{Atamturk:2018,Belotti:2015,Bonami:2015,Hijazi:2011,Hijazi:2014,Tawarmalani:2002}.
\blue{Finally, this paper has connections with distributionally robust optimization, in that the primal-dual pair of sharp formulations presented can be viewed as equivalent to the discrete sup-sup duality result \cite{chen2018distributionally} for the marginal distribution model \cite{haneveld1986robustness,natarajan2009persistency,Weiss}.}

\subsection{Starting assumptions and notation} \label{sec:assumptions}
We use the following notational conventions throughout the paper.
\begin{itemize}
    \item The nonnegative orthant: $\bbR_{\geq 0} \defeq \Set{x \in \bbR | x \geq 0}$.
    \item The \blue{$n$}-dimensional simplex: $\Delta^{\blue{n}} \defeq \Set{ x \in \bbR^{\blue{n}}_{\geq 0} | \sum_{i=1}^{\blue{n}} x_i = 1}$.
    \item The set of integers from $1$ to $n$: $\llbracket n \rrbracket = \{1, \ldots, n\}$.
    \item ``Big-$M$'' coefficients: $M^+(f;D) \defeq \max_{x \in D} f(x)$ and $M^-(f;D) \defeq \min_{x \in D} f(x)$.
    \item The dilation of a set: if $z \in \bbR_{\geq 0}$ and $D \subseteq \bbR^\eta$, then $z \cdot D \defeq \Set{ z x | x \in D}$.\footnote{Note that if $D = \Set{x \in \mathbb{R}^\eta | Ax \leq b}$ is polyhedral, then $z \cdot D = \Set{x \in \mathbb{R}^\eta | Ax \leq bz}$.}
\end{itemize}

Furthermore, we will make the following simplifying assumptions.

\begin{assumption}
    The input domain $D$ is a bounded polyhedron.
\end{assumption}

While a bounded input domain assumption will make the formulations and analysis considerably more difficult than the unbounded setting (see \cite{Atamturk:2018} for a similar phenomenon), it ensures that standard MIP representability conditions are satisfied (e.g. \cite[Section 11]{Vielma:2015}). Furthermore, variable bounds are natural for many applications (for example in verification problems), and are absolutely essential for ensuring reasonable dual bounds.

\begin{assumption} \label{ass:amphibious}
    Each neuron is \emph{irreducible}: for any $k \in \llbracket d \rrbracket$, there exists some $x \in D$ where $f^k(x) > f^\ell(x)$ for each $\ell \neq k$.
\end{assumption}

Observe that if a neuron is not irreducible, this means that it is unnecessarily complex, and one or more of the affine functions can be completely removed. Moreover, the assumption can be verified in polynomial time via $d$ LPs by checking if $\max_{x,\Delta} \Set{\Delta | x \in D,\ \Delta \leq f^k(x) - f^\ell(x)\ \forall \ell \neq k} > 0$ for each $k \in \llbracket d \rrbracket$. In the special case where $d = 2$ (e.g. ReLU) and $D$ is a box, this can be done in linear time.  Finally, if the assumption does not hold, it will not affect the validity of the formulations or cuts derived in this work, though certain results pertaining to non-redundancy or facet-defining properties may no longer hold.

\section{Motivating example: The ReLU nonlinearity over a box domain} \label{sec:relu}

\blue{
Since the work of Glorot et al.~\cite{glorot2011deep}, the ReLU neuron has become the workhorse of deep learning models. Despite the simplistic ReLU structure, neural networks formed by ReLU neurons are easy to train, reason about, and can model any continuous piecewise linear function \cite{arora2016understanding}.
Moreover, they can approximate any continuous, non-linear function to an arbitrary precision under a bounded width \cite{hanin2017universal}.

Accordingly, the general problem of maximizing (or minimizing) the output of a trained ReLU network is NP-hard \cite{Katz:2017}.
Nonetheless, in this section we present the strongest possible MIP formulations for a single ReLU neuron without continuous auxiliary variables. As discussed at the end of Section~\ref{sec:preliminaries} and corroborated in the computational study in Section~\ref{sec:computational}, this leads to faster solve times for the entire ReLU network in practice.
}

\subsection{A big-$M$ formulation} \label{ssec:relu-standard-formulations}

To start, we will consider the ReLU in the simplest possible setting: where the input is univariate. Take the two-dimensional set $\gr(\ReLu{}; [l,u])$, where $[l,u]$ is some interval in $\bbR$ containing zero. It is straightforward to construct an ideal formulation for this univariate ReLU.

\begin{proposition}
    An ideal formulation for $\gr(\emph{\ReLu{}}; [l,u])$ is:
\begin{subequations} \label{eqn:univariate-max-big-M}
\begin{align}
    y &\geq x \label{eqn:univariate-max-big-M-1} \\
    y &\leq x - l(1-z) \\
    y &\leq uz \\
    (x,y,z) &\in \bbR \times \bbR_{\geq 0} \times [0,1] \label{eqn:univariate-max-big-M-4} \\
    z &\in \{0,1\}.
\end{align}
\end{subequations}
\end{proposition}
\begin{proof}
    Follows from inspection, or as a special case of Proposition~\ref{prop:ideal-relu} to be presented in Section~\ref{ssec:relu-box}.
\qed \end{proof}

A more realistic setting would have an \blue{$\eta$-variate} ReLU nonlinearity whose input is some \blue{$\eta$-variate} affine function $f : [L,U] \to \bbR$ \blue{where $L,U\in \bbR^\eta$ and $[L,U]\defeq\Set{ x \in \bbR^\eta | L_i \leq  x_i \leq U_i\  \forall i\in \llbracket \eta \rrbracket} $} \blue{(i.e. the $\eta$-variate ReLU nonlinearity given by $\ReLu{} \circ f$)}. The \blue{ \emph{box-input}  $[L,U]$} corresponds to known (finite) bounds on each component, which can typically be efficiently computed via interval arithmetic or other standard methods.

Observe that we can model the \blue{graph} of the  \emph{\blue{$\eta$-variate} ReLU \blue{neuron}} as a simple composition of \blue{the graph of} a univariate ReLU \blue{activation function} and an \blue{$\eta$-variate} affine function:
\begin{equation}
  \blue{\gr\left(\ReLu{} \circ f;[L,U]\right)} =  \Set{ (x,y)\blue{\in\bbR^{\eta+1}} | \begin{array}{c}
        (f(x),y) \in \gr\left(\ReLu{}; [\blue{m^-,m^+}]\right) \\
        L \leq x \leq U
    \end{array}
    },
\end{equation}
\blue{where $m^-\defeq M^-(f; [L,U])$ and $m^+\defeq M^+(f; [L,U])$}.
Using formulation~\eqref{eqn:univariate-max-big-M} as a submodel, we can write a formulation for the ReLU over a box domain as:
\begin{subequations} \label{eqn:relu-big-M}
\begin{align}
    y &\geq f(x) \\
    y &\leq f(x) - M^-(f;[L,U]) \cdot (1-z) \label{eqn:relu-big-M-2} \\
    y &\leq M^+(f;[L,U]) \cdot z \label{eqn:relu-big-M-3} \\
    (x,y,z) &\in [L,U] \times \bbR_{\geq 0} \times [0,1] \\
    z &\in \{0,1\}.
\end{align}
\end{subequations}
This is the approach taken recently in the bevy of papers referenced in Section~\ref{sec:lit-review}. Unfortunately, after the composition with the affine function $f$ over a box input domain, this formulation is no longer sharp.

\begin{example} \label{ex:relu-big-M}
    If $f(x) = x_1 + x_2 - 1.5$, formulation \eqref{eqn:relu-big-M} for $\gr(\ReLu{} \circ f; [0,1]^2)$ is
    \begin{subequations}
    \begin{align}
        y &\geq x_1 + x_2 - 1.5 \label{eqn:relu-example-1} \\ 
        y &\leq x_1 + x_2 - 1.5 + 1.5(1-z) \\
        y &\leq 0.5z \\
        (x,y,z) &\in [0,1]^2\times \bbR_{\geq 0} \times [0,1] \label{eqn:relu-example-5} \\
        z &\in \{0,1\}.
    \end{align}
    \end{subequations}
    The point $(\hat{x},\hat{y},\hat{z}) = ((1,0),0.25,0.5)$ is feasible for the LP relaxation (\ref{eqn:relu-example-1}-\ref{eqn:relu-example-5}). However, observe that the inequality $y \leq 0.5x_2$ is valid for $\gr(\ReLu{} \circ f; [0,1]^2)$, but is violated by $(\hat{x},\hat{y})$. Therefore, the formulation does not offer an exact convex relaxation (and, hence, is not ideal). See Figure~\ref{fig:relu} for an illustration: on the left, of the big-$M$ formulation projected to $(x,y)$-space, and on the right, the tightest possible convex relaxation.
\end{example}

\tikzstyle{yzx} = [
  x={(1.2*.9625cm, 1.2*.9625cm)},
  y={(1.2*2.5cm, 0cm)},
  z={(0cm, 1.2*2.5cm)},
]

\begin{figure}
    \centering
    \begin{tikzpicture}[yzx]
        \draw [->, dashed, line width=1] (0,0,0) -- (1.2,0,0);
        \draw [->, dashed, line width=1] (0,0,0) -- (0,1.2,0);
        \draw [->, dashed, line width=1] (0,0,0) -- (0,0,0.7);
        \node[above right] at (1.2,-.075,0) {$x_2$};
        \node[right] at (0,1.2,0) {$x_1$};
        \node[above] at (0,0,.7) {$y$};
        \coordinate (LL) at (0,0,0);
        \coordinate (UL) at (1,0,0);
        \coordinate (LU) at (0,1,0);
        \coordinate (UU) at (1,1,0.5);
        \coordinate (UM) at (1,0.5,0);
        \coordinate (MU) at (0.5,1,0);
        \coordinate (E1) at (1,0,1/4);
        \coordinate (E2) at (0,1,1/4);

        \draw [fill=gray!80] (LL) -- (UL) -- (UM) -- (MU) -- (LU) -- cycle;
        \draw [fill=gray!80] (UU) -- (UM) -- (MU) -- cycle;
        \draw (LL) -- (E2) -- (UU) -- (E1) -- cycle;
        \draw (LL) -- (UL) -- (E1) -- cycle;
        \draw (LL) -- (LU) -- (E2) -- cycle;
        
        \begin{scope}[canvas is yz plane at x=0]
            \draw [fill] (1,1/4) circle [radius=.025];
        \end{scope}
    \end{tikzpicture} \hspace{2em}
    \begin{tikzpicture}[yzx]
        \draw [->, dashed, line width=1] (0,0,0) -- (1.2,0,0);
        \draw [->, dashed, line width=1] (0,0,0) -- (0,1.2,0);
        \draw [->, dashed, line width=1] (0,0,0) -- (0,0,0.7);
        \node[above right] at (1.2,-0.075,0) {$x_2$};
        \node[right] at (0,1.2,0) {$x_1$};
        \node[above] at (0,0,.7) {$y$};
        \coordinate (LL) at (0,0,0);
        \coordinate (UL) at (1,0,0);
        \coordinate (LU) at (0,1,0);
        \coordinate (UU) at (1,1,0.5);
        \coordinate (UM) at (1,0.5,0);
        \coordinate (MU) at (0.5,1,0);
        
        \draw [fill=gray!80] (LL) -- (UL) -- (UM) -- (MU) -- (LU) -- cycle;
        \draw [fill=gray!80] (UU) -- (UM) -- (MU) -- cycle;
        \draw (LL) -- (UL) -- (UU) -- cycle;
        \draw (LL) -- (LU) -- (UU) -- cycle;
        
        \begin{scope}[canvas is yz plane at x=0]
            \draw [fill] (1,1/4) circle [radius=.025];
        \end{scope}
    \end{tikzpicture}
    \caption{For $f(x) = x_1 + x_2 - 1.5$: \textbf{(Left)} $\Conv(\gr(\ReLu{} \circ f; [0,1]^2))$, and \textbf{(Right)} the projection of the big-$M$ formulation \eqref{eqn:relu-big-M} to $(x,y)$-space, where we mark the point $(x,y) = ((1,0),0.25)$ that is not in the convex hull, but is valid for the projection of the big-$M$ LP relaxation (\ref{eqn:relu-example-1}-\ref{eqn:relu-example-5}).}
    \label{fig:relu}
\end{figure}
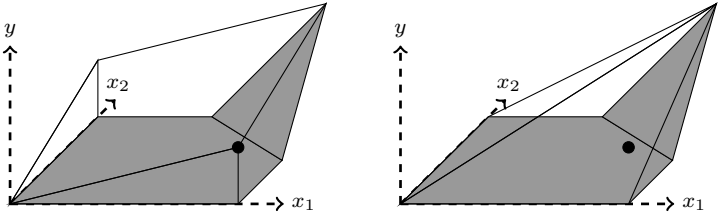

Moreover, the integrality gap of \eqref{eqn:relu-big-M} can be arbitrarily bad, even in fixed dimension $\eta$.

\begin{example}
Fix $\gamma \in \bbR_{\geq 0}$ and even $\eta \in \bbN$.
Take the affine function $f(x) = \sum_{i=1}^\eta x_i$, the input domain $[L,U] = [-\gamma,\gamma]^\eta$, and $\hat{x} = \gamma \cdot (1,-1,\cdots,1,-1)$ as a scaled vector of alternating positive and negative ones. We can check that $(\hat{x},\hat{y},\hat{z}) = (\hat{x},\frac{1}{2}\gamma\eta,\frac{1}{2})$ is feasible for the LP relaxation of the big-$M$ formulation \eqref{eqn:relu-big-M}. Additionally, $f(\hat{x}) = 0$, and for any $\tilde{y}$ such that $(\hat{x},\tilde{y}) \in \Conv(\gr(\ReLu{} \circ f; [L,U]))$, then $\tilde{y} = 0$ necessarily. Therefore, there exists a fixed point $\hat{x}$ in the input domain where the tightest possible convex relaxation (for example, from a sharp formulation) is exact, but the big-$M$ formulation deviates from this value by at least $\frac{1}{2}\gamma\eta$.
\end{example}

Intuitively, this example suggests that the big-$M$ formulation can be particularly weak around the boundary of the input domain, as it cares only about the value $f(x)$ of the affine function, and not the particular input value $x$.

\subsection{An ideal extended formulation} \label{ssec:ideal-extended}

It is possible to produce an ideal \emph{extended} formulation for the ReLU neuron by introducing a modest number of auxiliary continuous variables:
\begin{subequations} \label{eqn:relu-ideal-extended}
\begin{align}
    (x,y) &= (x^0,y^0) + (x^1,y^1) \label{eqn:relu-ideal-extended-1} \\
    y^0 &= 0 \geq w \cdot x^0 + b(1-z) \label{eqn:relu-ideal-extended-2} \\
    y^1 &= w \cdot x^1 + bz \geq 0 \label{eqn:relu-ideal-extended-3} \\
    L(1-z) &\leq x^0 \leq U(1-z)\label{eqn:relu-ideal-extended-4} \\
    Lz &\leq x^1 \leq Uz \label{eqn:relu-ideal-extended-5} \\
    z &\in \{0,1\},
\end{align}
\end{subequations}
This is the standard ``multiple choice'' formulation for piecewise linear functions~\cite{Vielma:2009a}, which can also be derived from techniques due to Balas~\cite{Balas:1985,Balas:1998}.

Although the multiple choice formulation offers the tightest possible convex relaxation for a single neuron, it requires a copy $x^0$ of the input variables (the copy $x^1$ can be eliminated using the equations \eqref{eqn:relu-ideal-extended-1}). This means that when this formulation is applied to every neuron in the network to formulate $\NN{}$, the total number of continuous variables required is $m_0 + \sum_{i=1}^{r} (m_{i-1}+1)m_{i}$, where $m_i$ is the number of neurons in layer $i$. In contrast, the big-$M$ formulation requires only $m_0 + \sum_{i=1}^r m_i$ continuous variables to formulate the entire network. As we will see in Section~\ref{sec:computational}, the quadratic growth in size of the extended formulation can quickly become burdensome. Additionally, a folklore observation in the MIP community is that multiple choice formulations tend to not perform as well as expected in simplex-based branch-and-bound algorithms, likely due to degeneracy introduced by the block structure of the formulation~\cite{Vielma:2018}.

\subsection{An ideal MIP formulation without auxiliary continuous variables} \label{ssec:ideal-relu}

In this work, our most broadly useful contribution is the derivation of an ideal MIP formulation for the ReLU nonlinearity over a box domain that is non-extended; that is, it does not require additional auxiliary variables as in formulation \eqref{eqn:relu-ideal-extended}. We informally summarize this main result as follows.


\medskip
\begin{theo*} \label{thm:informal}
    There exists an explicit ideal nonextended formulation for the  \blue{$\eta$-variate} ReLU nonlinearity over a box domain, i.e. it requires only a single auxiliary binary variable. It has an exponential (in $\eta$) number of inequality constraints, each of which are facet-defining. However, it is possible to separate over this family of inequalities in time scaling linearly in $\eta$.
\end{theo*}
\medskip

We defer the formal statement and proof to Section~\ref{ssec:relu-box}, for after we have derived the requisite machinery. This is also the main result for the extended abstract version of this work~\cite{Anderson:2019}, where it is derived through alternative means.


\section{Our general machinery: Formulations for the maximum of $d$ affine functions} \label{sec:max}

We will state our main structural results in the following generic setting. Take the maximum operator $\Max{}(v_1,\ldots,v_d) = \max_{i=1}^d v_i$ over $d$ scalar inputs. We study the composition of this nonlinearity with $d$ affine functions $f^i : D \to \bbR$ with $f^i(x) = w^i \cdot x + b^i$, all sharing some bounded polyhedral domain $D$:
\[
    \Smax \defeq \gr(\Max{} \circ (f^1,\ldots,f^d); D) \equiv \Set{ (x,y) \in D \times \bbR | y = \max_{i=1}^d f^i(x) },
\]
This setting subsumes the ReLU over box input domain presented in Section~\ref{sec:relu} as a special case with $d=2$, $f^2(x) = 0$, and $D = [L,U]$. It also covers a number of other settings arising in modern deep learning, either by making $D$ more complex (e.g. one-hot encodings for categorical features), or by increasing $d$ (e.g. max pooling neurons often used in convolutional networks for image classification tasks~\cite{Boureau:2010}, or in maxout networks~\cite{Goodfellow:2013}).

In this section, we present structural results that characterize the \emph{Cayley embedding}~\cite{Huber:2000,Vielma:2016,Vielma:2018,Weibel:2007} of $\Smax$. Take the set
\[
    \Scayley \defeq \bigcup_{k=1}^d \Set{ (x,f^k(x),{\bf e}^k) | x \in D_{|k}},
\]
where ${\bf e}^k$ is the unit vector where the $k$-th element is 1, and for each $k \in \llbracket d \rrbracket$,
\begin{align*}
D_{|k} &\defeq \Set{ x \in D | k \in \arg\max_{\ell=1}^d f^\ell(x)} \equiv \Set{x \in D | w^k \cdot x + b^k \geq w^\ell \cdot x + b^\ell \ \forall \ell \neq k }
\end{align*}
is the portion of the input domain $D$ where $f^k$ attains the maximum. The \emph{Cayley embedding} of $\Smax$ is the convex hull of this set: $\Rcayley \defeq \Conv(\Scayley)$.

The following straightforward observation holds directly from definition.

\begin{observation}
The set $\Rcayley$ is a bounded polyhedron, and an ideal formulation for $\Smax$ is given by the system $\Set{(x,y,z) \in \Rcayley | z \in \{0,1\}^d}$.
\end{observation}

Therefore, if we can produce an explicit inequality description for $\Rcayley$, we immediately have an ideal MIP formulation for $\Smax$. Indeed, we have already seen an extended representation in \eqref{eqn:relu-ideal-extended} when $d=2$, which we now state in the more general setting (projecting out the superfluous copies of $y$).

\begin{proposition}\label{prop:extended-formulation-cayley}
An ideal MIP formulation for $\Smax$ is:
    \begin{subequations}\label{eqn:extended-formulation}
    \begin{alignat}{2}
    (x,y) &= \sum_{k=1}^d (\tx^k, w^k \cdot \tx^k + b^kz_k) \label{eqn:extended-formulation-1} \\
    \tx^k &\in z_k \cdot D_{|k} &\quad &\forall k \in \llbracket d \rrbracket \label{eqn:extended-formulation-4} \\
    z &\in \Delta^d \label{eqn:extended-formulation-5} \\
    z &\in \{0,1\}^d.
    \end{alignat}
    \end{subequations}
    Denote its LP relaxation by $\Rextended = \Set{(x,y,z,\tx^1,\ldots,\tx^k) | (\ref{eqn:extended-formulation-1}-\ref{eqn:extended-formulation-5})}$. Then $\Proj_{x,y,z}(\Rextended) = \Rcayley$.
\end{proposition}
\blue{
\begin{proof}
Follows directly from results in \cite{Balas:1985,Balas:1998}.
\qed\end{proof}}
Although this formulation is ideal and polynomially-sized, this extended formulation can exhibit poor practical performance, as noted in Section~\ref{ssec:ideal-extended} and corroborated in the computational experiments in Section~\ref{sec:computational}. Observe that, from definition, constraint \eqref{eqn:extended-formulation-4} for a given $k \in \llbracket d \rrbracket$ is equivalent to the set of constraints $\tx^k \in z_k \cdot D$ and $w^k \cdot \tx^k + b^kz_k \geq w^\ell \cdot \tx^k + b^\ell z_k$ for each $\ell \neq k$.

\subsection{A recipe for constructing ideal formulations} \label{ssec:ideal-recipe}
Our goal in this section is to derive generic tools that allow us to build ideal formulations for $\Smax$ via the Cayley embedding.

\subsubsection{A primal characterization}

Our first structural result provides a characterization for the Cayley embedding. Although it is not an explicit polyhedral characterization, we will subsequently see how it can be massaged into a more practically amenable form.

Take the system
\begin{subequations}\label{eqn:max_ideal_set_primal}
\begin{align}
    y &\leq \bar{g}(x,z)\label{eqn:max_ideal_set_primal-1}\\
    y &\geq \ubar{g}(x,z)\label{eqn:max_ideal_set_primal-2}\\
    (x,y,z) &\in D \times \mathbb{R} \times \Delta^d,\label{eqn:max_ideal_set_primal-3}
\end{align}
\end{subequations}
where
\begin{align*}
    \bar{g}(x,z) \defeq \max_{\tx^1, \ldots, \tx^d}\Set{\sum_{k=1}^d w^k \cdot \tx^k + b^k z_k | \begin{array}{cc} x = \sum_k \tx^k & \\ \tx^k \in z_k \cdot D_{|k} &\ \forall k \in \llbracket d \rrbracket \end{array}}\hspace{0.45em}  \\
    \ubar{g}(x,z) \defeq \min_{\tx^1, \ldots, \tx^d}\Set{\sum_{k=1}^d w^k \cdot \tx^k + b^k z_k | \begin{array}{cc} x = \sum_k \tx^k & \\ \tx^k \in z_k \cdot D_{|k} &\ \forall k \in \llbracket d \rrbracket \end{array}},
\end{align*}
and define the set $\Rideal \defeq \Set{ (x,y,z) | \eqref{eqn:max_ideal_set_primal}}$. \blue{Note that $\Rideal$ can be thought of as \eqref{eqn:extended-formulation} with the $\tx^k$ variables implicitly projected out.}

\begin{proposition}\label{prop:max_ideal_set_primal}
    The set $\Rideal$ is polyhedral, and $\Rideal = \Rcayley$.
\end{proposition}
\begin{proof}
By Proposition~\ref{prop:extended-formulation-cayley}, it suffices to show that $\Rideal = \Proj_{x,y,z}(\Rextended)$. We start by observing that, as $\bar{g}$ (respectively $\ubar{g})$ is concave (resp. convex) in its imputs as it is the value function of a linear program (cf. \cite[Theorem 5.1]{Bertsimas:1997}). Therefore, the set of points satisfying \eqref{eqn:max_ideal_set_primal} is convex.

Let $(\hat{x}, \hat{y}, \hat{z})$ be an extreme point of $\Rideal$, which exists as $\Rideal$ is convex. Then it must satisfy either \eqref{eqn:max_ideal_set_primal-1} or \eqref{eqn:max_ideal_set_primal-2} at equality, as otherwise it is a convex combination of $(\hat{x}, \hat{y} - \epsilon, \hat{z})$ and $(\hat{x}, \hat{y} + \epsilon, \hat{z})$ for some $\epsilon > 0$. Take $\tx^1, \ldots, \tx^d$ that optimizes $\bar{g}(x,z)$ or $\ubar{g}(x,z)$, depending on which constraint is satisfied at equality. Then $(\hat{x}, \hat{y}, \hat{z}, \tx^1, \ldots, \tx^d) \in \Rextended$. In other words, $\text{ext}(\Rideal) \subseteq \Proj_{x,y,z}(\Rextended)$, and thus $\Rideal \subseteq \Proj_{x,y,z}(\Rextended)$ by convexity.

Conversely, let $(\hat{x},\hat{y},\hat{z},\tx^1,\ldots,\tx^d) \in \Rextended$. Then $\hat{y} = \sum_{k=1}^d (w^k \cdot \tx^k + b^k\hat{z}_k) \leq \bar{g}(\hat{x}, \hat{z})$, as $(\{\tx^k\}_{k=1}^d)$ is feasible for the optimization problem in $\bar{g}(x, z)$. 
Likewise, $\hat{y} \geq \ubar{g}(\hat{x}, \hat{z})$, and \eqref{eqn:max_ideal_set_primal-3} is trivially satisfied. Therefore, $(\hat{x}, \hat{y}, \hat{z}) \in \Rideal$.

Polyhedrality of $\Rideal$ then follows as $\Rcayley$ is itself a polyhedron.
\qed\end{proof}

Note that the input domain constraint $x \in D$ is implied by the constraints (\ref{eqn:max_ideal_set_primal-1})--(\ref{eqn:max_ideal_set_primal-2}), and therefore do not need to be explicitly included in this description. However, we include it here in our description for clarity. Moreover, observe that $\bar{g}$ is a function from $D \times \Delta^d$ to $\mathbb{R} \cup \{-\infty\}$, since the optimization problem may be infeasible, but is always bounded from above since $D$ is bounded. Likewise, $\ubar{g}$ is a function from $D \times \Delta^d$ to $\mathbb{R} \cup \{+\infty\}$.


\subsubsection{A dual characterization}\label{sec:ideal_dual}

From Proposition~\ref{prop:max_ideal_set_primal}, we can derive a more useful characterization by applying Lagrangian duality to the LPs describing the envelopes $\bar{g}(x,z)$ and $\ubar{g}(x,z)$.



\begin{proposition} \label{prop:max_ideal_set_dual}
The Cayley embedding $\Rcayley$ is equal to all $(x,y,z)$ satisfying
\begin{subequations}\label{eqn:max_ideal_set_dual}
\begin{align}
    y \leq \bar{\alpha} \cdot x + \sum_{k=1}^d \left(\max_{x^k \in D_{|k}}\{(w^k - \bar{\alpha}) \cdot x^k\} + b^k\right)z_k&\quad\forall \bar{\alpha} \in \mathbb{R}^\eta\label{eqn:max_ideal_set_dual-1}\\
    y \geq \ubar{\alpha} \cdot x + \sum_{k=1}^d \left(\min_{x^k \in D_{|k}}\{(w^k - \ubar{\alpha}) \cdot x^k\} + b^k\right)z_k&\quad\forall \ubar{\alpha} \in \mathbb{R}^\eta\label{eqn:max_ideal_set_dual-2}\\
    (x,y,z) \in D \times \mathbb{R} \times \Delta^d.
\end{align}
\end{subequations}
\end{proposition}
\begin{proof}
Consider the upper bound inequalities for $y$ in Proposition~\ref{prop:max_ideal_set_primal}, that is, \eqref{eqn:max_ideal_set_primal-1}. Now apply the change of variables $x^k \leftarrow \frac{\tx^k}{z_k}$ for each $k \in \llbracket d \rrbracket$ and, for all $(x,z)$, take the Lagrangian dual of the optimization problem in $\bar{g}(x,z)$ with respect to the constraint $x = \sum_k x^k z_k$. Note that the duality gap is zero since the problem is an LP. We then obtain that
\begin{align*}
    \bar{g}(x,z) &= \min_{\bar{\alpha}} \max_{x^k \in D_{|k}} \sum_{k=1}^d \left(w^k \cdot x^k + b^k\right) z_k + \bar{\alpha} \cdot \left(x - \sum_{k=1}^d x^k z_k\right) \notag\\
    &= \min_{\bar{\alpha}}\ \bar{\alpha} \cdot x + \sum_{k=1}^d \left(\max_{x^k \in D_{|k}}\{(w^k - \bar{\alpha}) \cdot x^k\} + b^k\right)z_k.
\end{align*}

In other words, we can equivalently express~\eqref{eqn:max_ideal_set_primal-1} via the family of inequalities~\eqref{eqn:max_ideal_set_dual-1}. The same can be done with~\eqref{eqn:max_ideal_set_primal-2}, yielding the set of inequalities~\eqref{eqn:max_ideal_set_dual-2}. Therefore, $\Set{(x,y,z) | \eqref{eqn:max_ideal_set_dual}} = \Proj_{x,y,z}(\Rextended) = \Rcayley$.
\qed\end{proof}

This gives an exact outer description for the Cayley embedding in terms of an infinite number of linear inequalities. Despite this, the formulation enjoys a simple, interpretable form: we can view the inequalities as choosing coefficients on $x$ and individually tightening the coefficients on $z$ according to explicitly described LPs. \blue{Similar to the relationship between \eqref{eqn:extended-formulation} and \eqref{eqn:max_ideal_set_primal}, \eqref{eqn:max_ideal_set_dual} can be seen as a simplification of the standard cut generation LP for unions of polyhedra \cite{Balas:1985,Balas:1998}.} In later sections, we will see that this decoupling is helpful to simplify (a variant of) this formulation for special cases.

Separating a point $(\hat{x}, \hat{y}, \hat{z})$ over~\eqref{eqn:max_ideal_set_dual-1} can be done by evaluating $\bar{g}(\hat{x},\hat{z})$ in the form~\eqref{eqn:max_ideal_set_dual-1} (and in the analogous form of $\ubar{g}(\hat{x},\hat{z})$ for~\eqref{eqn:max_ideal_set_dual-2}). As typically done when using Lagrangian relaxation, this optimization problem can be solved via a subgradient or bundle method, where each subgradient can be computed by solving the inner LP in~\eqref{eqn:max_ideal_set_dual-1} for all $x^k \in D_{|k}$. Observe that any feasible solution $\bar{\alpha}$ for the optimization problem in~\eqref{eqn:max_ideal_set_dual-1} yields a valid inequality. However, this optimization problem is unbounded when $(\hat{x}, \hat{z}) \notin \Proj_{x,z}(\Rcayley)$ (i.e. when the primal form of $\bar{g}(\hat{x}, \hat{z})$ is infeasible). In other words, as illustrated in Figure~\ref{fig:upper_bound_functions}, $\bar{g}$ is an \emph{extended} real valued function such that $\bar{g}({x}, {z})\in \mathbb{R}\cup \{-\infty\}$, so care must be taken to avoid numerical instabilities when separating a point $(\hat{x}, \hat{z})$ where $\bar{g}(\hat{x}, \hat{z})=-\infty$.\footnote{As is standard in a Benders' decomposition approach, we can address this by adding a \emph{feasibility cut} describing the domain of $\bar{g}$ (the region where it is finite valued) instead of an \emph{optimality cut} of the form \eqref{eqn:max_ideal_set_dual-1}.}



\subsection{A recipe for constructing hereditarily sharp formulations}

Although Proposition~\ref{prop:max_ideal_set_dual} gives a separation-based way to optimize over $\Smax$, there are two potential downsides to this approach. First, it does not give us a explicit, finite description for a MIP formulation that we can directly pass to a MIP solver. Second, the separation problem requires optimizing over $D_{|k}$, which may be substantially more complicated than optimizing over $D$ (for example, if $D$ is a box).

Therefore, in this section we set our sights slightly lower and present a similar technique to derive \emph{sharp} MIP formulations for $\Smax$. Furthermore, we will see that our formulations trivially satisfy the hereditary sharpness property. In the coming sections, we will see how we can deploy these results in a practical manner, and study settings in which the simpler sharp formulation will also, in fact, be ideal.

\subsubsection{A primal characterization}

Consider the system
\begin{subequations}\label{eqn:max_sharp_set_primal}
\begin{align}
    y &\leq \bar{h}(x, z)\label{eqn:max_sharp_set_primal-1}\\
    y &\geq w^k \cdot x + b^k\quad\forall k \in \llbracket d \rrbracket \label{eqn:max_sharp_set_primal-2}\\
    (x,y,z) &\in D \times \mathbb{R} \times \Delta^d,
\end{align}
\end{subequations}
where
\begin{align*}
    \bar{h}(x, z) \defeq \max_{\tx^1, \ldots, \tx^d}\Set{\sum_{k=1}^d (w^k \cdot \tx^k + b^k z_k) | \begin{array}{cc} x = \sum_k \tx^k & \\ \tx^k \in z_k \cdot D &\ \forall k \in \llbracket d \rrbracket \end{array}}.
\end{align*}
Take the set $\Rsharp \defeq \Set{ (x,y,z) | \eqref{eqn:max_sharp_set_primal} }$.

It is worth dwelling on the differences between the systems~\eqref{eqn:max_ideal_set_primal} and \eqref{eqn:max_sharp_set_primal}. First, we have completely replaced the constraint \eqref{eqn:max_ideal_set_primal-2} with $d$ explicit linear inequalities \eqref{eqn:max_sharp_set_primal-2}. Second, when replacing $\bar{g}$ with $\bar{h}$ we have replaced the inner maximization over $D_{|k}$ with an inner maximization over $D$ (modulo constant scaling factors). As we will see in Section~\ref{ssec:max-cuts}, this is particularly advantageous when $D$ is trivial to optimize over (for example, a simplex or a box), allowing us to write these constraints in closed form, whereas optimizing over $D_{|k}$ may be substantially more difficult (i.e. requiring an LP solve).

Furthermore, we will show that while \eqref{eqn:max_sharp_set_primal} is not ideal, it is hereditarily sharp, and so in general may offer a strictly stronger relaxation than a standard sharp formulation. In particular, the formulation may be stronger than a sharp formulation constructed by composing a big-$M$ formulation, along with an exact convex relaxation in the $(x,y)$-space produced, for example, by studying the upper concave envelope of the function $\Max \circ (f^1,\ldots,f^d)$.

\begin{proposition} \label{prop:max_sharp_set_primal}
    The system describing $\Set{(x,y,z) \in \Rsharp | z \in \{0,1\}^d}$ is a hereditarily sharp MIP formulation of $\Smax$.
\end{proposition}


\begin{proof}
For the result, we must show four properties: polyhedrality of $\Rsharp$, validity of the formulation whose LP relaxation is $\Rsharp$, sharpness, and then hereditary sharpness. We proceed in that order.

To show polyhedrality, consider a fixed value $(\hat{x},\hat{y},\hat{z})$ feasible for \eqref{eqn:max_sharp_set_primal}, and presume that we express the domain via the linear inequality constraints $D = \Set{x | Ax \leq c}$. First, observe that due to \eqref{eqn:max_sharp_set_primal-1} and \eqref{eqn:max_sharp_set_primal-2}, $\bar{h}(\hat{x},\hat{z})$ is bounded from below. Now, using LP duality, we may rewrite
\begin{align*}
    \bar{h}(\hat{x},\hat{z}) &= \max_{\tx^1, \ldots, \tx^d}\Set{\sum_{k=1}^d w^k \cdot \tx^k | \begin{array}{cc} \hat{x} = \sum_k \tx^k & \\ A \tx^k \leq \hat{z}_k c &\ \forall k \in \llbracket d \rrbracket \end{array}} + \sum_{k=1}^d b^k \hat{z}_k \\
    &= \min_{(\alpha, \beta^1, \ldots, \beta^k) \in R} \Set{ \alpha \cdot \hat{x} + \sum_{k=1}^d c \cdot \beta^k \hat{z}_k} + \sum_{k=1}^d b^k\hat{z}_k,
\end{align*}
where $R$ is a polyhedron that is \emph{independent of $\hat{x}$ and $\hat{z}$}. 
Therefore, as (a) the above optimization problem is linear with $\hat{x}$ and $\hat{z}$ fixed, and (b) $\bar{h}(\hat{x},\hat{z})$ is bounded from below, we may replace $R$ with $\ext(R)$ in the above optimization problem.
In other words, $\bar{h}(\hat{x},\hat{z})$ is equal to the minimum of a finite number of alternatives which are affine in $\hat{x}$ and $\hat{z}$. Therefore, $\bar{h}$ is a concave continuous piecewise linear function, and so $\Rsharp$ is polyhedral.

To show validity, we must have that $\Proj_{x,y}\left(\Rsharp \cap (\bbR^{\eta} \times \bbR \times \{0,1\}^d)\right) = \Smax$.
Observe that if $(\hat{x},\hat{y},\hat{z}) \in \Rsharp \cap (\bbR^{\eta} \times \bbR \times \{0,1\}^d)$, then $\hat{z} = {\bf e}^\ell$ for some $\ell \in \llbracket d \rrbracket$, and
\begin{align*}
    \bar{h}(\hat{x},\hat{z}) &= \max_{\tx^1, \ldots, \tx^d}\Set{\sum_{k=1}^d w^k \cdot \tx^k + b^\ell | \begin{array}{cc} \hat{x} = \sum_k \tx^k & \\ \tx^\ell \in D & \\ \tx^k = {\bf 0}^\eta &\forall k \neq \ell \end{array}} = w^\ell \hat{x} + b^\ell,
\end{align*}
where the first equality follows as $\tx^k = {\bf 0}^\eta$ for each $k \neq \ell$ (recall that if $D$ is bounded, then $0 \cdot D = \Set{x \in \bbR^\eta | Ax \leq 0} = \{{\bf 0}^\eta\}$).
Along with \eqref{eqn:max_sharp_set_primal-2}, this implies that $\hat{y} = w^\ell \cdot \hat{x} + b^\ell$, and that $\hat{y} \geq w^k \cdot \hat{x} + b^k$ for each $k \neq \ell$, giving the result.

To show sharpness, we must prove that $\Proj_{x,y}(\Rsharp) = \Conv(\Smax)$. First, recall from Proposition~\ref{prop:extended-formulation-cayley} that $\Conv(\Smax) = \Proj_{x,y}(\Rextended)$; thus, we state our proof in terms of $\Rextended$.
We first show that $\Proj_{x,y}(\Rextended) \subseteq \Proj_{x,y}(\Rsharp)$. Take $(\hat{x}, \hat{y}, \hat{z}, \{\hat{x}^k\}_{k=1}^d) \in \Rextended$. Then $\hat{y} = \sum_{k=1}^d (w^k \cdot \hat{x}^k + b^k\hat{z}_k) \leq \bar{h}(\hat{x}, \hat{z})$, as $(\{\hat{x}^k\}_{k=1}^d)$ is feasible for the optimization problem in $\bar{h}(x, z)$. It also holds that $\hat{y} \geq w^k \cdot \hat{x} + b^k$ for all $k \in \llbracket d \rrbracket$ and $\hat{x} \in D$ directly from the definition of $\Smax$, giving the result.

Next, we show that $\Proj_{x,y}(\Rsharp) \subseteq \Proj_{x,y}(\Rextended)$. This proof is similar to the proof of Proposition~\ref{prop:max_ideal_set_primal}, except that we choose $z$ that simplifies the constraints. It suffices to show that $\text{ext}(\Proj_{x,y}(\Rsharp)) \subseteq \Proj_{x,y}(\Rextended)$. Let $(\hat{x}, \hat{y}) \in \text{ext}(\Proj_{x,y}(\Rsharp))$. Define $\bar{h}(x) \defeq \max_{z} \Set{\bar{h}(x, z) | z \in \Delta^d}$. Then either $(\hat{x}, \hat{y})$ satisfies $\hat{y} = \bar{h}(\hat{x})$, or it satisfies~\eqref{eqn:max_sharp_set_primal-2} at equality for some $k \in \llbracket d \rrbracket$, since otherwise $(\hat{x}, \hat{y})$ is a convex combination of the points $(\hat{x}, \hat{y} - \epsilon)$ and $(\hat{x}, \hat{y} + \epsilon)$ feasible for $\Proj_{x,y}(\Rsharp)$ for some $\epsilon > 0$. We show that in either case, $(\hat{x}, \hat{y}) \in \Proj_{x,y}(\Rextended)$.

\underline{Case 1:} Suppose that for some $j \in \llbracket d \rrbracket$, $(\hat{x}, \hat{y})$ satisfies the corresponding inequality in \eqref{eqn:max_sharp_set_primal-2} at equality; that is, $\hat{y} = w^j \cdot \hat{x} + b^j$. Then the point $(\hat{x}, \hat{y}, \mathbf{e}_j, \{\hat{x}^k\}_{k=1}^d) \in \Rextended$, where $\hat{x}^j = x$ and $\hat{x}^\ell = 0$ if $\ell \neq j$. Hence, $(\hat{x}, \hat{y}) \in \text{proj}_{x,y}(\Rextended)$.

\underline{Case 2:} Suppose $(\hat{x}, \hat{y})$ satisfies $\hat{y} = \bar{h}(\hat{x})$. Let $\hat{z}$ be an optimal solution for the optimization problem defining $\bar{h}(\hat{x})$, and $\{\hat{x}^k\}_{k=1}^d$ be an optimal solution for $\bar{h}(\hat{x},\hat{z})$. By design, $(\hat{x}, \hat{y}, \hat{z}, \{\hat{x}^k\}_{k=1}^d)$ satisfies all constraints in $\Rextended$, except potentially constraint~\eqref{eqn:extended-formulation-4}.

We show that constraint~\eqref{eqn:extended-formulation-4} is satisfied as well. Suppose not for contradiction; that is, $w^k \cdot \hat{x}^k + b^k \hat{z}_k < w^\ell \cdot \hat{x}^k + b^\ell \hat{z}_k$ for some pair $k, \ell \in \llbracket d \rrbracket, \ell \neq k$. Consider the solution $(\{\bar{x}^k\}_{k=1}^d, \bar{z})$ identical to $(\{\hat{x}^k\}_{k=1}^d, \hat{z})$ except that $\bar{z}_k = 0$, $\bar{x}^k = {\bf 0}^\eta$, $\bar{z}_\ell = \hat{z}_\ell + \hat{z}_k$, and $\bar{x}^\ell = \hat{x}^\ell + \hat{x}^k$. By inspection, this solution is feasible for $\bar{h}(\hat{x})$. The objective value changes by $- (w^k \cdot \hat{x}^k + b^k \hat{z}_k) + (w^\ell \cdot \hat{x}^k + b^\ell \hat{z}_k) > 0$, contradicting the optimality of $(\{\hat{x}^k\}_{k=1}^d, \hat{z})$. Therefore, $(\hat{x}, \hat{y}, \hat{z}, \{\hat{x}^k\}_{k=1}^d) \in \Rextended$, and thus $(\hat{x}, \hat{y}) \in \Proj_{x,y}(\Rextended)$.

Finally, we observe that hereditary sharpness follows from the definition of $\bar{h}$. In particular, fixing any $z_k = 0$ implies that $\tx^k = {\bf 0}^\eta$ in the maximization problem defining $\bar{h}$. In other words, the variables $\tx^k$ and $z_k$ drop completely from the maximization problem defining $\bar{h}$, meaning that it is equal to the corresponding version of $\bar{h}$ with the function $k$ completely dropped as input. Additionally, if any $z_k = 1$, then $z_\ell = 0$ for each $\ell \neq k$ since $z \in \Delta^d$, and hence $\tx^\ell = {\bf 0}^\eta$. In this case, $\bar{h}(x,z) = w^k \cdot x + b^k$, which gives the result.
\qed\end{proof}



\subsubsection{A dual characterization}

We can apply a duality-based approach to produce an (albeit infinite) linear inequality description for the set $\Rsharp$, analogous to Section~\ref{sec:ideal_dual}. 

\begin{proposition}\label{prop:max_sharp_set_dual}
The set $\Rsharp$ is equal to all $(x,y,z)$ such that
\begin{subequations}\label{eqn:max_sharp_set_dual}
\begin{align}
    y \leq \alpha \cdot x + \sum_{k=1}^d \left(\max_{x^k \in D}\{(w^k - \alpha) \cdot x^k\} + b^k\right)z_k&\quad\forall \alpha \in \mathbb{R}^\eta\label{eqn:max_sharp_set_dual-1}\\
    y \geq w^k \cdot x + b^k&\quad\forall k \in \llbracket d \rrbracket\label{eqn:max_sharp_set_dual-2}\\
    (x,y,z) \in D \times \mathbb{R} \times \Delta^d. \label{eqn:max_sharp_set_dual-3}
\end{align}
\end{subequations}
\end{proposition}
\begin{proof}

Follows in an analogous manner as in Proposition~\ref{prop:max_ideal_set_dual}.
\qed\end{proof}

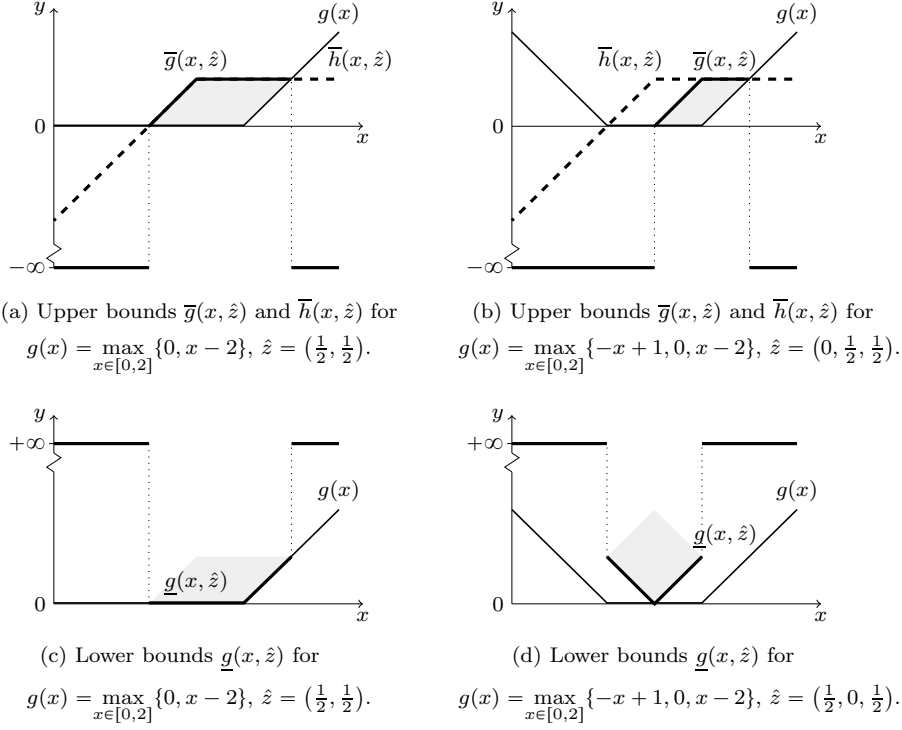
\begin{figure}
    \centering
    \begin{subfigure}[t]{0.44\textwidth}
    \begin{tikzpicture}[scale=1.25]
        \fill[opacity=0.25, fill=lightgray] (1,0) -- (1.5,0.5) -- (2.5,0.5) -- (2,0) -- cycle;

        \draw [->, line width=0.25] (0,0) -- (3.25,0);
        \draw [->, line width=0.25] (0,-1.25) -- (0,1.25);
        \node[left] at (0,1.25) {$y$};
        \node[below] at (3.25,0) {$x$};
        \node[left] at (0, 0) {$0$};
        \coordinate (P1) at (0,0.01);
        \coordinate (P2) at (2,0.01);
        \coordinate (P3) at (3,1);

        \draw[semithick] (P1) -- (P2) -- (P3);

        \draw[very thick] (1,0) -- (1.5,0.5) -- (2.5,0.5);
        \draw[very thick] (0,-1.5) -- (1,-1.5);
        \draw[very thick] (2.5,-1.5) -- (3,-1.5);
        \draw[dotted] (1,-1.5) -- (1,0);
        \draw[dotted] (2.5,-1.5) -- (2.5,0.5);

        \draw[very thick,dashed,dash phase=2pt] (0,-1) -- (1.5,0.5) -- (3,0.5);

        \draw[line width=0.25] (0,-1.25) -- (0.075,-1.3) -- (-0.075,-1.4) -- (0,-1.45) -- (0,-1.5);
        \draw[line width=0.25] (-0.05,-1.5) -- (0,-1.5);
        \node[left] at (0, -1.5) {$-\infty$};

        \node[above] at (1.5,0.5) {$\bar{g}(x, \hat{z})$};
        \node[above right] at (2.8,0.5) {$\bar{h}(x, \hat{z})$};
        
        \node[above] at (3, 1) {\blue{$g(x)$}};
    \end{tikzpicture}
    \caption{Upper bounds $\bar{g}(x,\hat{z})$ and $\bar{h}(x,\hat{z})$ for\\[0.5em] \centering$\blue{g(x) =\!}\displaystyle\max_{x \in [0,2]}\{0, x-2\}$, $\hat{z} = \left(\frac{1}{2},\frac{1}{2}\right)$.}\label{fig:upper_bound_functions_a}
    \end{subfigure} \hspace{2em}
    \begin{subfigure}[t]{0.49\textwidth}
    \begin{tikzpicture}[scale=1.25]
        \fill[opacity=0.25, fill=lightgray] (1.5,0) -- (2,0.5) -- (2.5,0.5) -- (2,0) -- cycle;

        \draw [->, line width=0.25] (0,0) -- (3.25,0);
        \draw [->, line width=0.25] (0,-1.25) -- (0,1.25);
        \node[left] at (0,1.25) {$y$};
        \node[below] at (3.25,0) {$x$};
        \node[left] at (0, 0) {$0$};
        \coordinate (P1) at (0,1);
        \coordinate (P2) at (1,0.01);
        \coordinate (P3) at (2,0.01);
        \coordinate (P4) at (3,1);

        \draw[semithick] (P1) -- (P2) -- (P3) -- (P4);

        \draw[very thick] (1.5,0) -- (2,0.5) -- (2.5,0.5);
        \draw[very thick] (0,-1.5) -- (1.5,-1.5);
        \draw[very thick] (2.5,-1.5) -- (3,-1.5);
        \draw[dotted] (1.5,-1.5) -- (1.5,0);
        \draw[dotted] (2.5,-1.5) -- (2.5,0.5);

        \draw[very thick, dashed,dash phase=2pt] (0,-1) -- (1.5,0.5) -- (3,0.5);

        \draw[line width=0.25] (0,-1.25) -- (0.075,-1.3) -- (-0.075,-1.4) -- (0,-1.45) -- (0,-1.5);
        \draw[line width=0.25] (-0.05,-1.5) -- (0,-1.5);
        \node[left] at (0, -1.5) {$-\infty$};

        \node[above] at (2.25,0.5) {$\bar{g}(x, \hat{z})$};
        \node[above] at (1.25,0.5) {$\bar{h}(x, \hat{z})$};
        
        \node[above] at (3, 1) {\blue{$g(x)$}};
    \end{tikzpicture}
    \caption{Upper bounds $\bar{g}(x,\hat{z})$ and $\bar{h}(x,\hat{z})$ for \\[0.5em] \centering$\blue{g(x) =\!}\displaystyle\max_{x \in [0,2]}\{-x+1, 0, x-2\}$, $\hat{z} = \left(0,\frac{1}{2},\frac{1}{2}\right)$.}\label{fig:upper_bound_functions_b}
    \end{subfigure}\vspace{1em}
    \begin{subfigure}[t]{0.44\textwidth}
    \begin{tikzpicture}[scale=1.25]
        \fill[opacity=0.25, fill=lightgray] (1,0) -- (1.5,0.5) -- (2.5,0.5) -- (2,0) -- cycle;

        \draw [->, line width=0.25] (0,0) -- (3.25,0);
        \draw [line width=0.25] (0,0) -- (0,1.4);
        \node[left] at (0,2) {$y$};
        \node[below] at (3.25,0) {$x$};
        \node[left] at (0, 0) {$0$};
        \coordinate (P1) at (0,0.01);
        \coordinate (P2) at (2,0.01);
        \coordinate (P3) at (3,1);

        \draw[semithick] (P1) -- (P2) -- (P3);

        \draw[very thick] (1,0.01) -- (2,0.01) -- (2.5,0.5);
        \draw[very thick] (0,1.7) -- (1,1.7);
        \draw[very thick] (2.5,1.7) -- (3,1.7);
        \draw[dotted] (1,0) -- (1,1.7);
        \draw[dotted] (2.5,0.5) -- (2.5,1.7);

        \draw[line width=0.25] (0,1.4) -- (-0.075,1.45) -- (0.075,1.55) -- (0,1.6);
        \draw[->,line width=0.25] (0,1.6) -- (0,2);
        \draw[line width=0.25] (-0.05,1.7) -- (0,1.7);
        \node[left] at (0, 1.7) {$+\infty$};

        \node[above] at (1.5,0) {$\ubar{g}(x, \hat{z})$};
        
        \node[above] at (3, 1) {\blue{$g(x)$}};
    \end{tikzpicture}
    \caption{Lower bounds $\ubar{g}(x,\hat{z})$ for\\[0.5em] \centering$\blue{g(x) =\!}\displaystyle\max_{x \in [0,2]}\{0, x-2\}$, $\hat{z} = \left(\frac{1}{2},\frac{1}{2}\right)$.}\label{fig:lower_bound_functions_a}
    \end{subfigure} \hspace{2em}
    \begin{subfigure}[t]{0.49\textwidth}
    \begin{tikzpicture}[scale=1.25]
        \fill[opacity=0.25, fill=lightgray] (1,0.5) -- (1.5,0) -- (2,0.5) -- (1.5,1) -- cycle;

        \draw [->, line width=0.25] (0,0) -- (3.25,0);
        \draw [line width=0.25] (0,0) -- (0,1.4);
        \node[left] at (0,2) {$y$};
        \node[below] at (3.25,0) {$x$};
        \node[left] at (0, 0) {$0$};
        \coordinate (P1) at (0,1);
        \coordinate (P2) at (1,0.01);
        \coordinate (P3) at (2,0.01);
        \coordinate (P4) at (3,1);

        \draw[semithick] (P1) -- (P2) -- (P3) -- (P4);

        \draw[very thick] (1,0.5) -- (1.5,0) -- (2,0.5);
        \draw[very thick] (0,1.7) -- (1,1.7);
        \draw[very thick] (2,1.7) -- (3,1.7);
        \draw[dotted] (1,0.5) -- (1,1.7);
        \draw[dotted] (2,0.5) -- (2,1.7);

        \draw[line width=0.25] (0,1.4) -- (-0.075,1.45) -- (0.075,1.55) -- (0,1.6);
        \draw[->,line width=0.25] (0,1.6) -- (0,2);
        \draw[line width=0.25] (-0.05,1.7) -- (0,1.7);
        \node[left] at (0, 1.7) {$+\infty$};

        \node[above] at (2.25,0.5) {$\ubar{g}(x, \hat{z})$};
        
        \node[above] at (3, 1) {\blue{$g(x)$}};
    \end{tikzpicture}
    \caption{Lower bounds $\ubar{g}(x,\hat{z})$ for\\[0.5em] \centering$\blue{g(x) =\!}\displaystyle\max_{x \in [0,2]}\{-x+1, 0, x-2\}$, $\hat{z} = \left(\frac{1}{2},0,\frac{1}{2}\right)$.}\label{fig:lower_bound_functions_b}
    \end{subfigure}
        \caption{Examples of the functions $\bar{g}(x,z)$, $\bar{h}(x,z)$, and $\ubar{g}(x,z)$ defined in~\eqref{eqn:max_ideal_set_primal} and~\eqref{eqn:max_sharp_set_primal} with some fixed value for $z$. \blue{Here, $g(x)$ is the ``maximum'' function of interest, defined below each subfigure. Figures (a) and (c) depict the case $d = 2$ (maximum of two affine functions), while (b) and (d) illustrate when $d = 3$, each pair emphasizing upper and lower bounds respectively.} Note that $\bar{g}(x,z)$ and $\bar{h}(x,z)$ coincide in (a) for $x \in [1,2.5]$, and in (b) for $x \in [2,2.5]$. The thick solid lines represent $\bar{g}(x,z)$ in (a--b) and $\ubar{g}(x,z)$ in (c--d), whereas the dashed lines correspond to $\bar{h}(x,z)$. The thin solid lines represent $\Smax$ and the shaded region is the slice of $\Rcayley$ with $z=\hat{z}$.}
    \label{fig:upper_bound_functions}
\end{figure}

Figure~\ref{fig:upper_bound_functions} depicts slices of the functions $\bar{g}(x,z)$, $\ubar{g}(x,z)$, and $\bar{h}(x,z)$, created by fixing some value of $z$ and varying $x$. Observe that $\bar{g}(x,z)$ can be viewed as the largest value for $y$ such that $(x,y)$ can be written as a convex combination of points in the graph using convex multipliers $z$. Likewise, $\ubar{g}(x,z)$ can be interpreted as the minimum value for $y$. In $\bar{h}(x,z)$, we relax $D_{|k}$ to $D$, and thus we can interpret it similarly to $\bar{g}(x,z)$, except that we may take convex combinations of points constructed by evaluating the affine functions at any point in the domain, not only those where the given function attains the maximum.  Figure~\ref{fig:upper_bound_functions_b} shows that, in general, $\bar{h}(x,z)$ can be strictly looser than $\bar{g}(x,z)$ for $(x,z) \in \Proj_{x,z}(\Rcayley)$. A similar situation occurs for the lower envelopes as illustrated by Figure~\ref{fig:lower_bound_functions_b}. However, we prove in the next section that this does not occur for $d = 2$, along with other desirable properties in special cases.
\section{Simplifications to our machinery under common special cases} \label{sec:simplifications}

In this section we study how our dual characterizations in Propositions~\ref{prop:max_ideal_set_dual} and \ref{prop:max_sharp_set_dual} simplify under common special cases with the number of input affine functions $d$ and the input domain $D$.

\subsection{Simplifications when $d=2$} \label{sec:d-2}

When we consider taking the maximum of only two affine functions (i.e. $d=2$), we can prove that $\Rsharp$ is, in fact, ideal.

We start by returning to Proposition~\ref{prop:max_ideal_set_primal} and show that it can be greatly simplified when $d = 2$. We first show $\ubar{g}(x,z)$ can be replaced by the maximum of the affine functions as illustrated in Figure~\ref{fig:lower_bound_functions_a}, although it is not possible for $d > 2$ as seen in Figure~\ref{fig:lower_bound_functions_b}.

\begin{lemma}\label{lem:max_2_ideal_lb}
    If $d=2$, then at any values of $(x,z)$ where $\bar{g}(x,z)\ge\ubar{g}(x,z)$ (i.e. there exists a $y$ such that $(x,y,z) \in \Rcayley$), we have
    \begin{align*}
        \ubar{g}(x,z) = \max\{w^1 \cdot x + b^1, w^2 \cdot x + b^2\}.
    \end{align*}
\end{lemma}
\begin{proof}
For convenience, we work with the change of variables $x^k \leftarrow \frac{\tx^k}{z_k}$ for each $k \in \llbracket 2 \rrbracket$. Suppose without loss of generality that $x\in D_{|2}$, and consider any feasible solution $(x^1,x^2)$ to the optimization problem for $\ubar{g}(x,z)$, which is feasible by the assumption that $(x,z) \in \Proj_{x,z}(\Rcayley)$. We will show that $(w^1\cdot x^1+b^1)z_1+(w^2\cdot x^2+b^2)z_2\ge w^2\cdot x+b^2$. We assume that $z_2>0$, since otherwise $x=x^1\in D_{|1}\cap D_{|2}$ and the result is immediate.

Since $x=x^1z_1+x^2z_2$ and $z \in \Delta^2$, the line segment joining $x^1$ to $x^2$ contains $x$. Furthermore, since $x^1\in D_{|1}$ and $x^2\in D_{|2}$, this line segment also intersects the hyperplane $\Set{\hat{x}\in\mathbb{R}^\eta | w^1 \cdot \hat{x}+b^1=w^2 \cdot \hat{x}+b^2}$. Let $\hat{x}^1$ denote this point of intersection, and let $\hat{z}^1\in \Delta^2$ be such that $\hat{x}^1=x^1\hat{z}^1_1+x^2\hat{z}^1_2$. Since $x\in D_{|2}$, we know that $\hat{x}^1$ is closer to $x^1$ than $x$, i.e.\ $\hat{z}^1_1 \ge z_1$. Moreover, take the point $\hat{x}^2$ on this line segment such that $x=\hat{x}^1z_1+\hat{x}^2z_2$,
where $\hat{x}^2=x^1\hat{z}^2_1+x^2\hat{z}^2_2$ for some $\hat{z}^2 \in \Delta^2$. We have $\hat{z}^2_1 \le z_1$ since $\hat{x}^2$ is further away from $x^1$ than $x$.  Note that $\hat{x}^1\in D_{|1}\cap D_{|2}$ while $\hat{x}^2\in D_{|2}$, and thus $(\hat{x}^1, \hat{x}^2)$ is feasible.

It can be computed that $\hat{z}^2_1 = z_1\cdot\frac{\hat{z}^1_2}{z_2}$, which implies that $z_1=z_1(\hat{z}^1_1+\hat{z}^1_2)=z_1\hat{z}^1_1+z_2\hat{z}^2_1$ and $z_2=z_2(\hat{z}^2_1+\hat{z}^2_2)=z_1\hat{z}^1_2+z_2\hat{z}^2_2$. Using these two identities, we obtain
\begin{align*}
(w^1\cdot x^1+b^1)z_1+(w^2\cdot x^2+b^2)z_2 
= &\ (w^1\cdot x^1+b^1)(z_1\hat{z}^1_1+z_2\hat{z}^2_1)+(w^2\cdot x^2+b^2)(z_1\hat{z}^1_2+z_2\hat{z}^2_2) \\
= &\ ((w^1\cdot x^1+b^1)\hat{z}^1_1+(w^2\cdot x^2+b^2)\hat{z}^1_2)z_1 \\
& + ((w^1\cdot x^1+b^1)\hat{z}^2_1+(w^2\cdot x^2+b^2)\hat{z}^2_2)z_2 \\
= &\ (f(x^1)\hat{z}^1_1+f(x^2)\hat{z}^1_2)z_1+(f(x^1)\hat{z}^2_1+f(x^2)\hat{z}^2_2)z_2,
\end{align*}
where we let $f(\hat{x})$ denote the function $\max\{w^1\cdot\hat{x}+b^1,w^2\cdot\hat{x}+b^2\}$, recalling that $x^1\in D_{|1}$ and $x^2\in D_{|2}$.  Since $f(\hat{x})$ is convex, by Jensen's inequality the preceding expression is at least
$f(x^1\hat{z}^1_1+x^2\hat{z}^1_2)z_1+f(x^1\hat{z}^2_1+x^2\hat{z}^2_2)z_2$. The preceding expression equals $(w^2\cdot\hat{x}^1+b^2)z_1+(w^2\cdot\hat{x}^2+b^2)z_2$ by the definitions of $\hat{x}^1$ and $\hat{x}^2$, and the fact that they both lie in $D_{|2}$. Recalling the equation $x=\hat{x}^1z_1+\hat{x}^2z_2$ used to select $\hat{x}^2$ completes the proof.
\qed\end{proof}

Moreover, we show that when $d=2$ we can replace $\bar{g}$ with $\bar{h}$, as illustrated in Figure~\ref{fig:upper_bound_functions_a}. This property may not hold when $d > 2$ as shown in Figure~\ref{fig:upper_bound_functions_b}.

\begin{lemma}\label{lem:max_2_ideal_ub}
    If $d = 2$, then at any values of $(x,z)$ where $\bar{g}(x,z)\ge\ubar{g}(x,z)$ (i.e. there exists a $y$ such that $(x,y,z) \in \Rcayley$), we have
\begin{align}\label{eqn:max_2_ideal_ub}
\bar{g}(x,z) = \max_{\tx^1, \tx^2}\Set{w^1 \cdot \tx^1 + b^1 z_1 + w^2 \cdot \tx^2 + b^2 z_2 | \begin{array}{l} x = \tx^1 + \tx^2 \\ \tx^1 \in z_1 \cdot D \\ \tx^2 \in z_2 \cdot D \end{array}}.
\end{align}
\end{lemma}

\begin{proof}

We show that despite expanding the feasible set of the optimization problem in~\eqref{eqn:max_2_ideal_ub} by replacing $D_{|k}$ by $D$, its optimal value is no greater when $(x,z)$ is such that $\bar{g}(x,z)\ge\ubar{g}(x,z)$.  It suffices to show without loss of generality that $w^1 \cdot \tx^1 + b^1 z_1 \geq w^2 \cdot \tx^1 + b^2 z_1$ holds for any optimal $\tx^1, \tx^2$. By the assumption on $(x,z)$ and Lemma~\ref{lem:max_2_ideal_lb}, we have $\bar{g}(x,z) \geq w^2 \cdot x + b^2$, which implies the existence of some optimal $\tx^1, \tx^2$. That is, we have $w^2 \cdot (x - \tx^1) + b^2 z_2 + w^1 \cdot \tx^1 + b^1 z_1 \geq w^2 \cdot x + b^2$, which is equivalent to $w^1 \cdot \tx^1 + b^1 z_1 \geq w^2 \cdot \tx^1 + b^2 z_1$.
\qed\end{proof}

After observing that these simplifications are identical to those presented in Proposition~\ref{prop:max_sharp_set_dual}, we obtain the following corollary promised at the beginning of the section.

\begin{corollary} \label{cor:max_2_ideal_general}
    When $d=2$, $\Set{(x,y,z) \in \Rsharp | z \in \{0,1\}^d}$ is an ideal MIP formulation of $\Smax$.
\end{corollary}
\proof{}
Lemmas~\ref{lem:max_2_ideal_lb} and \ref{lem:max_2_ideal_ub} imply that $\Rsharp = \Rideal$, while Proposition~\ref{prop:max_ideal_set_primal} implies that $\Rideal = \Rcayley$, completing the chain and giving the result.
\qed\endproof


In later sections, we will study conditions under which we can produce an explicit inequality description for $\Rsharp$.

    



\subsection{Simplifications when $D$ is the product of simplices} \label{sec:product-of-simplices}

In this section, we consider another important special case: when the input domain is the Cartesian product of simplices. Indeed, the box domain case introduced in Section~\ref{sec:relu} can be viewed as a product of two-dimensional simplices, and we will also see in Section~\ref{ssec:one-hot-encoding} that this structure naturally arises in machine learning settings with categorical or discrete features.

When $D$ is the product of simplices, we can derive a finite representation for the the set \eqref{eqn:max_sharp_set_dual} (i.e. a finite representation for the infinite family of linear inequalities \eqref{eqn:max_sharp_set_dual-1}) through an elegant connection with the transportation problem. To do so, we introduce the following notation.
\begin{definition} \label{defn:prodSimpl}
Suppose the input domain is $D = \prod_{i=1}^\tau \Delta^{p_i}$, with $p_1+\cdots+p_{\tau}=\eta$. For notational simplicity, we re-organize the indices of $x$ and refer to its entries via $x_{i,j}$, where $i\in\llbracket\tau\rrbracket$ is the simplex index, and $j\in\llbracket p_i\rrbracket$ refers to the coordinate within simplex $i$.
The domain for $x$ is then
\begin{align} \label{eqn:simplexDef}
D=\Set{ ((x_{i,j})_{j=1}^{p_i})_{i=1}^{\tau} | (x_{i,j})_{j=1}^{p_i} \in \Delta^{p_i} \quad \forall i \in \llbracket \tau \rrbracket },
\end{align}
where the rows of $x$ correspond to each simplex. Correspondingly, we re-index the weights of the affine functions so that for each $k\in\llbracket d\rrbracket$, we have
$f^k(x)=\sum_{i=1}^{\tau}\sum_{j=1}^{p_i}w^k_{i,j}x_{i,j}+b^k$.
\end{definition}

Using the notation from Definition~\ref{defn:prodSimpl}, constraints~(\ref{eqn:max_sharp_set_dual-1}) can be written as
\begin{align*}
y\leq\sum_{i=1}^{\tau}\sum_{j=1}^{p_i}\alpha_{i,j}x_{i,j} + \sum_{k=1}^d \left(\max_{x^k \in D}\sum_{i=1}^{\tau}\sum_{j=1}^{p_i}(w^k_{i,j} - \alpha_{i,j})x^k_{i,j} + b^k\right)z_k\quad\forall\alpha\in\mathbb{R}^{\eta}.
\end{align*}
Since $D$ is a product of simplices, the maximization over $x^k\in D$ appearing in the right-hand side above is separable over each simplex $i \in \llbracket \tau \rrbracket$. Moreover, for each simplex $i$, the maximum value of $\sum_{j=1}^{p_i}(w^k_{ij} - \alpha_{ij})x^k_{ij}$, subject to the constraint $x^k \in D$, is obtained when $x^k_{ij}=1$ for some $j\in\llbracket p_i\rrbracket$.
Therefore, the family of constraints~(\ref{eqn:max_sharp_set_dual-1}) is equivalent to
\begin{align}
y &\leq\min_{\alpha}\left(\sum_{i=1}^{\tau}\sum_{j=1}^{p_i}\alpha_{i,j}x_{i,j} + \sum_{k=1}^d \left(\sum_{i=1}^{\tau}\max_{j=1}^{p_i}(w^k_{i,j} - \alpha_{i,j}) + b^k\right)z_k\right) \nonumber \\
&=\sum_{i=1}^{\tau}\min_{\alpha_{i,1},\ldots,\alpha_{i,p_i}}\left(\sum_{j=1}^{p_i}\alpha_{i,j}x_{i,j} + \sum_{k=1}^dz_k\cdot\max_{j=1}^{p_i}(w^k_{i,j} - \alpha_{i,j})\right)+\sum_{k=1}^db^kz_k. \label{eqn:sharp_dual_prod_simpl}
\end{align}

We show that the minimization problem in (\ref{eqn:sharp_dual_prod_simpl}), for any $i$, is equivalent to a transportation problem defined as follows.
\begin{definition} \label{def:trans}
For any values $x \in \Delta^p$ and $z \in \Delta^d$, and arbitrary weights $w^k_j\in\mathbb{R}$ for all $j\in\llbracket p\rrbracket$ and $k\in\llbracket d \rrbracket$, define the \textit{max-weight transportation problem} to be
\begin{alignat*}{2}
\trans(x,z;w^1,\ldots,w^d) \defeq \max_{\beta \geq 0}\Set{\sum_{k=1}^d\sum_{j=1}^pw^k_j\beta^k_j \;|\; \begin{alignedat}{3}\sum_{k=1}^d\beta^k_j &=x_j \quad\forall j\in\llbracket p\rrbracket \\
\sum_{j=1}^p\beta^k_j &=z_k \quad\forall k\in\llbracket d\rrbracket\end{alignedat}}
\end{alignat*}
\end{definition}
In the transportation problem, since $\sum_jx_j=1=\sum_kz_k$, it follows that $\beta^k_j \in [0,1]$, and so this value can be interpreted as the percent of total flow ``shipped" between $j$ and $k$.
The relation to~(\ref{eqn:sharp_dual_prod_simpl}) is now established through LP duality.

\begin{proposition} \label{prop:transportationDuality}
For any fixed $x \in \Delta^{p}$ and $z \in \Delta^d$,
\begin{align} \label{eqn:suppress_i}
    \min_{\alpha} \left( \sum_{j=1}^{p}\alpha_{j}x_{j} + \sum_{k=1}^dz_k\cdot\max_{j=1}^p(w^k_{j} - \alpha_{j}) \right) = \trans(x,z;w^1,\ldots,w^d).
\end{align}
Therefore, when $D$ is a product of simplices, the constraints \eqref{eqn:max_sharp_set_dual-1} can be replaced in \eqref{eqn:max_sharp_set_dual} with the single inequality
\begin{align} \label{eqn:unsupress_i}
y\le\sum_{i=1}^{\tau}\trans((x_{i,j})_{j=1}^{p_i},z;(w^1_{i,j})_{j=1}^{p_i},\ldots,(w^d_{i,j})_{j=1}^{p_i})+\sum_{k=1}^db^kz_k.
\end{align}
\end{proposition}

\begin{proof}
By using a variable $\gamma_{k}$ to model the value of $\max_{j=1}^{p}(w^k_{j} - \alpha_{j})$ for each $k\in\llbracket d\rrbracket$, the minimization problem on the LHS of~\eqref{eqn:suppress_i} is equivalent to
\begin{alignat*}{2}
\min_{\alpha,\gamma}\Set{\sum_{j=1}^{p}\alpha_{j}x_{j}+\sum_{k=1}^d\gamma_{k}z_k \;|\;
\gamma_{k} \ge w^k_{j}-\alpha_{j} \quad\forall k\in\llbracket d\rrbracket,j\in\llbracket p\rrbracket}
\end{alignat*}
which is a minimization LP with free variables $\alpha_j$ and $\gamma_k$.
Applying LP duality, this completes the proof of equation~\eqref{eqn:suppress_i}. The inequality
\eqref{eqn:unsupress_i} then arises by substituting equation~\eqref{eqn:suppress_i} into~\eqref{eqn:sharp_dual_prod_simpl}, for every simplex $i=1,\ldots,\tau$.
\qed\end{proof}

\subsection{Simplifications when both $d=2$ and $D$ is the product of simplices} \label{sec:d-2-product-of-simplices}

Proposition~\ref{prop:transportationDuality} shows that, when the input domain is a product of simplices, the tightest upper bound on $y$ can be computed through a series of transportation problems. We now leverage the fact that if either side of the transportation problem from Definition~\ref{def:trans} (i.e. $p$ or $d$) has only two entities, then it reduces to a simpler fractional knapsack problem.
Later, this will allow us to represent \eqref{eqn:sharp_dual_prod_simpl}, in either of the cases $d=2$ or $p_1=\cdots=p_{\tau}=2$, using an explicit finite family of linear inequalities in $x$ and $z$ which has a greedy linear-time separation oracle.
\begin{proposition} \label{prop:simplicesD2Explicit}
Given data $w^1, w^2 \in \bbR^p$, take $\tw_j = w^1_j-w^2_j$ for all $j\in\llbracket p\rrbracket$, and suppose the indices have been sorted so that
$\tw_1\le\cdots\le\tw_p$.
Then
\begin{align} \label{eqn:trans-is-fractional-knapsack}
\trans(x,z;w^1,w^2)=\min_{J=1}^{p}\left(\tw_Jz_1+\sum_{j=J+1}^{p}(\tw_j-\tw_J)x_j\right)+\sum_{j=1}^{p}w^2_jx_j.
\end{align}
Moreover, a $J \in \llbracket p \rrbracket$ that attains the minimum in the right-hand side of \eqref{eqn:trans-is-fractional-knapsack} can be found in $\mathcal{O}(p)$ time.
\end{proposition}

\begin{proof}
When $d=2$, the transportation problem becomes
\begin{alignat}{2} \label{eqn:trans-opt-proof}
\max_{\beta^1,\beta^2 \geq 0}\Set{\sum_{j=1}^p(w^1_j\beta^1_j+w^2_j\beta^2_j)\;|\;\beta^1_j+\beta^2_j =x_j \quad\forall j\in\llbracket p\rrbracket,\quad \sum_{j=1}^p\beta^1_j =z_1}
\end{alignat}
where the constraint $\sum_{j=1}^p\beta^2_j=z_2$ is implied because
$\sum_{j=1}^p\beta^2_j=\sum_{j=1}^p(x_j-\beta^1_j)=1-\sum_{j=1}^p\beta^1_j=1-z_1=z_2$.
Substituting $\beta^2_j=x_j-\beta^1_j$ for all $j \in \llbracket p \rrbracket$ and then omitting the superscript ``1'', \eqref{eqn:trans-opt-proof} becomes
\begin{alignat*}{2}
\max_{\beta}\Set{\sum_{j=1}^p(w^1_j-w^2_j)\beta_j+\sum_{j=1}^pw^2_jx_j\;|\;\sum_{j=1}^p\beta_j =z_1,\quad  0\le\beta_j \le x_j \quad \forall j\in\llbracket p\rrbracket}
\end{alignat*}
which is a fractional knapsack problem that can be solved greedily.

An optimal solution to the fractional knapsack LP above, assuming the sorting $w^1_1-w^2_1\le\cdots\le w^1_p-w^2_p$, can be computed greedily. Let $J\in\llbracket p\rrbracket$ be the \textit{maximum} index at which $\sum_{j=J}^px_j\ge z_1$.
We set $\beta_j = 0$ for all $j < J$, $\beta_J = z_1 - \sum_{j=J+1}^p x_j$, and $\beta_j = x_j$ for each $j > J$.
The optimal cost is
\begin{align*}
\sum_{j=J+1}^p(w^1_j-w^2_j)x_j+(w^1_J-w^2_J)\left(z_1-\sum_{j=J+1}^px_j\right)+\sum_{j=1}^pw^2_jx_j,
\end{align*}
which yields the desired expression after substituting $\tw_j=w^1_j-w^2_j$ for all $j\ge J$.
Moreover, the index $J$ above can be found in $\mathcal{O}(p)$ time by storing a running total for $\sum_{j=J}^p x_j$, completing the proof.
\qed\end{proof}

Observe that the $\mathcal{O}(p)$ runtime in Proposition~\ref{prop:simplicesD2Explicit} is non-trivial, as a na\"{i}ve implementation would run in time $\mathcal{O}(p^2)$, as the inner sum is linear in $p$.

Combining Propositions \ref{prop:transportationDuality} and \ref{prop:simplicesD2Explicit} immediately yields the following result.
\begin{corollary} \label{cor:simplicesD2Result}
Suppose that $d=2$ and that $D$ is a product of simplices.  Let $z = z_1 \equiv 1-z_2$.
For each simplex $i \in \llbracket \tau \rrbracket$, take $\tw_{i,j} = w^1_{i,j}-w^2_{i,j}$ for all $j=1,\ldots,p_i$ and relabel the indices so that $\tw_{i,1}\le\cdots\le\tw_{i,p_i}$.
Then, in the context of \eqref{eqn:max_sharp_set_dual}, the upper-bound constraints \eqref{eqn:max_sharp_set_dual-1} are equivalent to
\begin{align}
&y\le\sum_{i=1}^{\tau}\left(\tw_{i,J(i)}z+\sum_{j=J(i)+1}^{p_i}(\tw_{i,j}-\tw_{i,J(i)})x_{i,j}+\sum_{j=1}^{p_i}w^2_{i,j}x_{i,j}\right)+(b^1-b^2)z+b^2 \nonumber \\
&\hspace{9em}\forall\text{ mappings $J:\llbracket\tau\rrbracket\to\mathbb{Z}$ with $J(i)\in\llbracket p_i\rrbracket\ \forall i\in\llbracket\tau\rrbracket$.} \label{eqn:simplicesD2upperBound}
\end{align}
Moreover, given any point $(x,y,z)\in D\times\mathbb{R}\times[0,1]$, feasibility can be verified or a most violated constraint can be found in $\mathcal{O}(p_1+\cdots+p_{\tau})$ time.
\end{corollary}
Corollary~\ref{cor:simplicesD2Result} gives an explicit finite family of linear inequalities equivalent to \eqref{eqn:max_sharp_set_dual}. Moreover, we have already shown in Corollary~\ref{cor:max_2_ideal_general} that $\Rsharp$ yields an ideal formulation when $d=2$. Hence, we have an ideal nonextended formulation whose exponentially-many inequalities can be separated in $\mathcal{O}(p_1+\cdots+p_{\tau})$ time, where the initial sorting requires $\mathcal{O}(p_1\log p_1+\cdots+p_{\tau}\log p_{\tau})$ time. Note that this sorting can be avoided: we may instead solve the fractional knapsack problem in the separation via weighted median in $\mathcal{O}(p_1+\cdots+p_{\tau})$ time~\cite[Chapter 17.1]{Korte:2000}.

We can also show that none of the constraints in~\eqref{eqn:simplicesD2upperBound} are redundant.
\begin{proposition} \label{prop:simplicesD2FacetDefining}
Consider the polyhedron $P$ defined as the intersection of all halfspaces corresponding to the inequalities \eqref{eqn:simplicesD2upperBound}. Consider some arbitrary mapping $J : \llbracket \tau \rrbracket \to \bbZ$ with $J(i) \in \llbracket p_i \rrbracket$ for each $i \in \llbracket \tau \rrbracket$. Then the inequality in \eqref{eqn:simplicesD2upperBound} for $J$ is irredundant with respect to $P$. That is, removing the halfspace corresponding to mapping $J$ (note that this halfspace could also correspond to other mappings) from the description of $P$ will strictly enlarge the feasible set.

\end{proposition}
\begin{proof}
Fix a mapping $J$.  Consider the feasible points with $z=1/2$, and $x_{i,j}= \bbone[j = J(i)]$ for each $i$ and $j$.
At such points, the constraint corresponding to any mapping $J' \neq J$ in (\ref{eqn:simplicesD2upperBound}) is
\begin{align*}
y &\le\sum_{i=1}^{\tau}\left(\frac{\tw_{i,J'(i)}}{2}+\sum_{j=J'(i)+1}^{p_i}(\tw_{i,j}-\tw_{i,J'(i)})\mathbb{1}[j=J(i)]+\sum_{j=1}^{p_i}w^2_{i,j}\mathbb{1}[j=J(i)]\right)+\frac{b^1+b^2}{2} \\
&=\sum_{i=1}^{\tau}\left(\frac{\tw_{i,J'(i)}}{2}+(\tw_{i,J(i)}-\tw_{i,J'(i)})\mathbb{1}[J'(i)<J(i)]\right)+\sum_{i=1}^{\tau}w^2_{i,J(i)}+\frac{b^1+b^2}{2}.
\end{align*}
For any simplex $i$, recall that the indices are sorted so that $\tw_{i1}\le\cdots\le\tw_{ip_i}$.
Thus, if $J'(i)\ge J(i)$, then the expression inside the outer parentheses equals $\frac{\tw_{i,J'(i)}}{2} \geq \frac{\tw_{i,J(i)}}{2}$.
On the other hand, if $J'(i)<J(i)$, then the expression inside the outer parentheses can be re-written as $\frac{\tw_{i,J(i)}}{2}+\frac{\tw_{i,J(i)}-\tw_{i,J'(i)}}{2} \geq \frac{\tw_{i,J(i)}}{2}$.
Therefore, setting $J'(i)=J(i)$ for every simplex $i$ achieves the tightest upper bound in (\ref{eqn:simplicesD2upperBound}), which simplifies to
\begin{align} \label{eqn:facetTightestUpperBound}
y\le\sum_{i=1}^{\tau}\frac{\tw_{i,J(i)}}{2}+\sum_{i=1}^{\tau}w^2_{i,J(i)}+\frac{b^1+b^2}{2}.
\end{align}

Now, suppose that the same upper bound on $y$ is achieved by a mapping $J'$ such that $J'(i)\neq J(i)$ on a simplex $i$.
By the argument above, regardless of whether $J'(i)>J(i)$ or $J'(i)<J(i)$, the expression inside the outer parentheses can equal $\frac{\tw_{i,J(i)}}{2}$ only if $\tw_{i,J'(i)}=\tw_{i,J(i)}$.
In this case, inspecting the term inside the summation in \eqref{eqn:simplicesD2upperBound} for mappings $J$ and $J'$, we observe that regardless of the values of $x$ or $z$,
\begin{align*}
\tw_{i,J'(i)}z+\sum_{j=J'(i)+1}^{p_i}(\tw_{i,j}-\tw_{i,J'(i)})x_{i,j}=\tw_{i,J(i)}z+\sum_{j=J(i)+1}^{p_i}(\tw_{i,j}-\tw_{i,J(i)})x_{i,j}.
\end{align*}
Therefore, for a mapping $J'$ to achieve the tightest upper bound~(\ref{eqn:facetTightestUpperBound}), it must be the case that $\tw_{i,J'(i)}=\tw_{i,J(i)}$ on every simplex $i$, and so $J'$ and $J$ necessarily correspond to the same half-space in $D\times\mathbb{R}\times[0,1]$, completing the proof.
\qed\end{proof}


To close this section, we consider the setting where every simplex is 2-dimensional, i.e.\ $p_1=\cdots=p_{\tau}=2$, but the number of affine functions $d$ can be arbitrary.
Note that by contrast, the previous results (Proposition~\ref{prop:simplicesD2Explicit}, Corollary~\ref{cor:simplicesD2Result}, Proposition~\ref{prop:simplicesD2FacetDefining}) held in the setting where $d=2$ but $p_1,\ldots,p_{\tau}$ were arbitrary.
We exploit the symmetry of the transportation problem to immediately obtain the following analogous results, all wrapped up into Corollary~\ref{cor:simplicesP2}.

\begin{corollary} \label{cor:simplicesP2}
Given data $w^1, \ldots, w^d \in \bbR^2$, take $\tw^k = w^k_1 - w^k_2$ for all $k\in\llbracket d\rrbracket$, and suppose the indices have been sorted so that
$\tw^1\le\cdots\le\tw^d$.
Then
\begin{align} \label{eqn:p-2-trans-knapsack}
\trans(x,z;w^1,\ldots,w^d)=\min_{K=1}^d\left(\tw^Kx_1+\sum_{k=1}^{K}w^k_2z_k+\sum_{k=K+1}^{d}(w^k_1-\tw^K)z_k\right).
\end{align}
Moreover, a $K \in \llbracket d \rrbracket$ that attains the minimum in the right-hand side of \eqref{eqn:p-2-trans-knapsack} can be found in $\mathcal{O}(d)$ time.

Therefore, suppose that $D$ is a product of $\tau$ simplices of dimensions $p_1=\cdots=p_{\tau}=2$.
For each simplex $i\in\llbracket\tau\rrbracket$, take $\tw^k_i=w^k_{i,1}-w^k_{i,2}$ for all $k=1,\ldots,d$ and relabel the indices so that $\tw^1_i\le\cdots\le\tw^d_i$.
Then, in the context of \eqref{eqn:max_sharp_set_dual}, the upper-bound constraints \eqref{eqn:max_sharp_set_dual-1} are equivalent to
\begin{align}
y &\le\sum_{i=1}^{\tau}\left(\tw^{K(i)}_ix_{i,1}+\sum_{k=1}^{K(i)}w^k_{i,2}z_k+\sum_{k=K(i)+1}^{d}(w^k_{i,1}-\tw^{K(i)}_i)z_k\right)+\sum_{k=1}^db^kz_k \nonumber \\
&\hspace{18em}\forall\text{ mappings $K:\llbracket\tau\rrbracket\to\llbracket d\rrbracket$.} \label{eqn:p-2-trans-knapsack-full}
\end{align}
Furthermore, none of the constraints in~\eqref{eqn:p-2-trans-knapsack-full} are redundant.
\end{corollary}

Corollary~\ref{cor:simplicesP2} will be particularly useful in Section~\ref{ssec:max-cuts}, where it will allow us to derive sharp formulations for the maximum of $d$ affine functions over a box input domain, in analogy to \eqref{eqn:simplicesD2upperBound}.


\section{Applications of our machinery} \label{sec:applications}

We are now prepared to return to the concrete goal of this paper: building strong MIP formulations for nonlinearities used in modern neural networks.

\subsection{A hereditarily sharp formulation for \Max{} on box domains} \label{ssec:max-cuts}

We can now present a hereditarily sharp formulation, with exponentially-many constraints that can be efficiently separated, for the maximum of $d > 2$ affine functions over a shared box input domain. 
\begin{proposition}\label{prop:max_sharp}
    For each $k, \ell \in \llbracket d \rrbracket$, take
    \begin{align*}
        N^{\ell,k} &= \sum_{i=1}^\eta \max\{(w^k_i-w^\ell_i)L_i, (w^k_i-w^\ell_i)U_i)\}.
    \end{align*}
    A valid MIP formulation for $\gr(\emph{\Max{}} \circ (f^1,\ldots,f^d); [L,U])$ is
    \begin{subequations} \label{eqn:max_big_m}
    \begin{align}
    &y \leq w^\ell \cdot x + \sum_{k=1}^d \left(N^{\ell,k} + b^k\right)z_k \quad \forall \ell \in \llbracket d \rrbracket \label{eqn:max_big_m-1} \\
   &y \geq w^k \cdot x + b^k \quad\forall k \in \llbracket d \rrbracket \label{eqn:max_d_box-2} \\
    &(x,y,z) \in [L,U] \times \mathbb{R} \times \Delta^d \label{eqn:max_d_box-3}\\
    &z \in \{0,1\}^d.
    \end{align}
    \end{subequations}
    Moreover, a hereditarily sharp formulation is given by \eqref{eqn:max_big_m}, along with the constraints
\begin{align}
&y \leq \sum_{i=1}^{\eta}\left(w^{I(i)}_ix_i+\sum_{k=1}^d\max\{(w^{k}_i-w^{I(i)}_i)L_i,(w^{k}_i-w^{I(i)}_i)U_i\}z_k\right)+\sum_{k=1}^db^kz_k\nonumber\\ &\hspace{19em}\forall \text{ mappings } I : \llbracket \eta \rrbracket \to \llbracket d \rrbracket \label{eqn:max_d_box-1}
\end{align}
Furthermore, none of the constraints~\eqref{eqn:max_d_box-2} and~\eqref{eqn:max_d_box-1} are redundant.
\end{proposition}
\begin{proof}
The result directly follows from Corollary~\ref{cor:simplicesP2} if we make a change of variables to transform the box domain $[L,U]$ to the domain $(\Delta^2)^{\eta}$, where the $x$-coordinates are given by $x_{i,1},x_{i,2}\ge0$ over $i\in\llbracket\eta\rrbracket$, with $x_{i,1}+x_{i,2}=1$ for each simplex $i$.
For each $i$, let $w^k_{i,1}=w^k_iU_i,w^k_{i,2}=w^k_iL_i$, and let $\sigma_i:\llbracket d\rrbracket\to\llbracket d\rrbracket$ be a permutation such that
$w^{\sigma_i(1)}_{i,1}-w^{\sigma_i(1)}_{i,2}\le\cdots\le w^{\sigma_i(d)}_{i,1}-w^{\sigma_i(d)}_{i,2}$.
Fix a mapping $I:\llbracket\eta\rrbracket\to\llbracket d\rrbracket$ for~\eqref{eqn:p-2-trans-knapsack-full}. We make the change of variables $\xi_i \leftarrow (U_i-L_i)x_{i,1}+L_i$ and rewrite~\eqref{eqn:p-2-trans-knapsack-full} from Corollary~\ref{cor:simplicesP2} to take the form~\eqref{eqn:max_d_box-1}, as follows:
\begin{align}
y&\leq \sum_{i=1}^{\eta}\Bigg(w^{\sigma_i(I(i))}_i(U_i-L_i)x_{i,1}+\sum_{k=1}^{I(i)}w^{\sigma_i(k)}_iL_iz_k+\sum_{k=I(i)+1}^{d}(w^{\sigma_i(k)}_iU_i-w^{\sigma_i(I(i))}_i(U_i-L_i))z_k\Bigg)+\sum_{k=1}^db^kz_k,\nonumber\\
&=\sum_{i=1}^{\eta}\left(w^{\sigma_i(I(i))}_i\xi_i+\sum_{k=1}^{I(i)}(w^{\sigma_i(k)}_i-w^{\sigma_i(I(i))}_i)L_iz_k+\sum_{k=I(i)+1}^{d}(w^{\sigma_i(k)}_i-w^{\sigma_i(I(i))}_i)U_iz_k\right)+\sum_{k=1}^db^kz_k \nonumber \\
&=\sum_{i=1}^{\eta}\left(w^{\sigma_i(I(i))}_i\xi_i+\sum_{k=1}^d\max\{(w^{\sigma_i(k)}_i-w^{\sigma_i(I(i))}_i)L_i,(w^{\sigma_i(k)}_i-w^{\sigma_i(I(i))}_i)U_i\}z_k\right)+\sum_{k=1}^db^kz_k \nonumber \\
&=\sum_{i=1}^{\eta}\left(w^{\sigma_i(I(i))}_i\xi_i+\sum_{k=1}^d\max\{(w^{k}_i-w^{\sigma_i(I(i))}_i)L_i,(w^{k}_i-w^{\sigma_i(I(i))}_i)U_i\}z_k\right)+\sum_{k=1}^db^kz_k, \label{eqn:removeSigma}
\end{align}
where the first equality holds because $\xi_i-(U_i-L_i)x_{i,1}=L_i$ for each $i$ and $\sum_{k=1}^d z_k=1$, the second equality holds because $(w^{\sigma_i(k)}_i-w^{\sigma_i(I(i))}_i)L_i\ge(w^{\sigma_i(k)}_i-w^{\sigma_i(I(i))}_i)U_i$ if $k\le I(i)$ (and a reverse argument can be made if $k>I(i)$), and the third equality holds as each $\sigma_i:\llbracket d\rrbracket\to\llbracket d\rrbracket$ is a bijection. Now, since~(\ref{eqn:removeSigma}) is taken over all mappings $I:\llbracket\eta\rrbracket\to\llbracket d\rrbracket$, it is equivalent to replace $\sigma_i(I(i))$ with $I(i)$ in (\ref{eqn:removeSigma}). This yields~\eqref{eqn:max_d_box-1} over $\xi$ instead of $x$.

The lower bound in the context of Corollary~\ref{cor:simplicesP2} can be expressed as $y \geq \sum_{i=1}^{\eta}(w^k_{i,1}x_{i,1}+w^k_{i,2}x_{i,2}) + b^k,\ \forall k \in \llbracket d \rrbracket$. To transform it to~\eqref{eqn:max_d_box-2} over $\xi$, observe that $w^k_{i,1}x_{i,1}+w^k_{i,2}x_{i,2}$ can be re-written as $w^k_i(U_ix_{i,1}+L_i(1-x_{i,1}))=w^k_i\xi_i$.

Finally, note that this transformation of variables is a bijection since each $x_{i,1}$ was allowed to range over $[0,1]$, and thus each $\xi_i$ is allowed to range over $[L_i,U_i]$. Hence the new formulation over $(\xi,y,z)\in[L,U]\times\mathbb{R}\times\Delta^d$ is also sharp, yielding the desired result.

The irredundancy of~\eqref{eqn:max_d_box-1} also follows from Corollary~\ref{cor:simplicesP2}. For the irredundancy of~\eqref{eqn:max_d_box-2}, fix $k$ and take a point $(\hat{x}, \hat{y}, \hat{z})$ where $f^k(\hat{x}) > f^\ell(\hat{x})$ for all $\ell \neq k$, which exists by Assumption~\ref{ass:amphibious}. Thus, since the inequalities~\eqref{eqn:max_d_box-2} are the only ones bounding $y$ from below, the point $(\hat{x}, \hat{y} - \epsilon, \hat{z})$ for some $\epsilon > 0$ is satisfied by all constraints except~\eqref{eqn:max_d_box-2} corresponding to $k$, and therefore it is not redundant.
\qed\end{proof}

Observe that the constraints \eqref{eqn:max_big_m-1} are a special case of \eqref{eqn:max_d_box-1} for those constant mappings $I(i) = \ell$ for each $i \in \llbracket \eta \rrbracket$. We note in passing that this big-$M$ formulation~\eqref{eqn:max_big_m} may be substantially stronger than existing formulations appearing in the literature. For example, the tightest version of the formulation of Tjeng et al.~\cite{Tjeng:2017} is equivalent to \eqref{eqn:max_big_m} with the coefficients $N^{\ell,k}$ replaced with
\[
    N^{\ell,k} = b^\ell - b^k + \max\limits_{t \neq \ell} \left( \sum_{i=1}^\eta \left(\max\{w^t_iL_i,w^t_iU_i\} - \min\{w^\ell_iL_i,w^\ell_iU_i\}\right) \right).
\]
Note in particular that as the inner maximization and minimization are completely decoupled, and that the outer maximization in the definition of $N^{\ell,k,+}$ is completely independent of $k$.

In Appendix~\ref{app:max_big_m}, we generalize~\eqref{eqn:max_big_m} to provide a valid formulation for arbitrary polytope domains. We next emphasize that the hereditarily sharp formulation from Proposition~\ref{prop:max_sharp} is particularly strong when $d = 2$.

\begin{corollary}\label{cor:max_2_ideal}
The formulation given by \eqref{eqn:max_big_m} and \eqref{eqn:max_d_box-1} is ideal when $d = 2$. Moreover, the constraints in the families~\eqref{eqn:max_d_box-1} and~\eqref{eqn:max_d_box-2} are facet-defining.
\end{corollary}
\begin{proof}
Idealness follows directly from Corollary~\ref{cor:max_2_ideal_general}. Since the constraints~\eqref{eqn:max_d_box-1} and~\eqref{eqn:max_d_box-2} are irredundant and the formulation is ideal, they must either be facet-defining or describe an implied equality. Given that the equality $\sum_{k=1}^d z_k = 1$ appears in~\eqref{eqn:max_d_box-3}, it suffices to observe that the polyhedron defined by~\eqref{eqn:max_big_m} and~\eqref{eqn:max_d_box-1} has dimension $\eta + d$, which holds under Assumption~\ref{ass:amphibious}. 
\qed\end{proof}

We can compute a most-violated inequality from the family~\eqref{eqn:max_d_box-1} efficiently.

\begin{proposition} \label{prop:max_sharp_sep}
Consider the family of inequalities~\eqref{eqn:max_d_box-1}. Take some point $(\hat{x},\hat{y},\hat{z}) \in [L,U] \times \bbR \times \Delta^d$. If any constraint in the family is violated at the given point, a most-violated constraint can be constructed by selecting $\hat{I} : \llbracket \eta \rrbracket \to \llbracket d \rrbracket$ such that
\begin{equation}
        \hat{I}(i) \in \arg\min_{\ell\in\llbracket d\rrbracket} \left( w^{\ell}_i\hat{x}_i+\sum_{k=1}^d\max\{(w^{k}_i-w^{\ell}_i)L_i,(w^{k}_i-w^{\ell}_i)U_i\}\hat{z}_k \right) \label{eqn:max_box_d_separate}
\end{equation}
for each $i \in \llbracket \eta \rrbracket$. Moreover, if the weights $w_i^k$ are sorted on $k$ for each $i \in \llbracket \eta \rrbracket$, this can be done in $\mathcal{O}(\eta d)$ time.
\end{proposition}
\begin{proof}
Follows directly from Corollary~\ref{cor:simplicesP2}, which says that the minimization problem~\eqref{eqn:max_box_d_separate} can be solved in $\mathcal{O}(d)$ time for any $i\in\llbracket\eta\rrbracket$.
\qed\end{proof}
Note that na\"{i}vely, the minimization problem~\eqref{eqn:max_box_d_separate} would take $\mathcal{O}(d^2)$ time, because one has to check every $\ell\in\llbracket d\rrbracket$, and then sum over $k\in\llbracket d\rrbracket$ for every $\ell$. However, if we instead pre-sort the weights $w^1_i,\ldots,w^d_i$ for every $i\in\llbracket\eta\rrbracket$ in $\mathcal{O}(\eta d \log d)$ time, we can use Corollary~\ref{cor:simplicesP2} to run efficiently separate via a linear search. We note, however, that this pre-sorting step can potentially be obviated by solving the fractional knapsack problems appearing as a weighted median problem, which can be solved in $\mathcal{O}(d)$ time.

\subsection{The ReLU over a box domain} \label{ssec:relu-box}

We can now present the results promised in Theorem~\ref{thm:informal}. In particular, we derive a non-extended ideal formulation for the ReLU nonlinearity, stated only in terms of the original variables $(x,y)$ and the single additional binary variable $z$. Put another way, it is the strongest possible tightening that can be applied to the big-$M$ formulation \eqref{eqn:relu-big-M}, and so matches the strength of the multiple choice formulation without the growth in the number of variables remarked upon in Section~\ref{ssec:ideal-extended}. Notationally, for each $i \in \llbracket \eta \rrbracket$ take
\[
    \breve{L}_i = \begin{cases} L_i & \text{ if } w_i \geq 0 \\ U_i & \text{ if } w_i < 0 \end{cases} \quad \text{ and } \quad \breve{U}_i = \begin{cases} U_i & \text{ if } w_i \geq 0 \\ L_i & \text{ if } w_i < 0 \end{cases}.
\]


\begin{proposition} \label{prop:ideal-relu}
    Take some affine function $f(x) = w \cdot x + b$ over input domain $D = [L,U]$. The following is an ideal MIP formulation for $\gr(\emph{\ReLu{}} \circ f; [L,U])$:
    \begin{subequations} \label{eqn:ideal-single-relu}
    \begin{align}
        y &\leq \sum_{i \in I} w_i(x_i - \breve{L}_i(1-z)) + \left(b + \sum_{i \not\in I} w_i\breve{U}_i\right)z \quad \forall I \subseteq \llbracket\eta\rrbracket \label{eqn:ideal-single-relu-2} \\
        y &\geq w \cdot x + b \label{eqn:ideal-single-relu-1} \\
        (x,y,z) &\in [L,U] \times \bbR_{\geq 0} \times [0,1] \label{eqn:ideal-single-relu-3} \\
        z &\in \{0,1\}. \label{eqn:ideal-single-relu-4}
    \end{align}
    \end{subequations}
    Furthermore, each inequality in \eqref{eqn:ideal-single-relu-2} and \eqref{eqn:ideal-single-relu-1} is facet-defining.
\end{proposition}
\begin{proof}\footnote{Alternatively, a constructive proof of validity and idealness using Fourier--Motzkin elimination is given in the extended abstract of this work~\cite[Proposition 1]{Anderson:2019}.}
This result is a special case of Corollary~\ref{cor:max_2_ideal}. Observe that~\eqref{eqn:ideal-single-relu} is equivalent to~\eqref{eqn:max_big_m} and~\eqref{eqn:max_d_box-1} with $w^1_i=w_i$ and $w^2_i=0$ for all $i \in \llbracket \eta \rrbracket$, $b^1=b$, $b^2=0$, $z_1=z$, and $z_2=1-z$. The constraints~\eqref{eqn:ideal-single-relu-2} are found by setting $I=\Set{i \in \llbracket\eta\rrbracket | \hat{I}(i)=1}$ for each mapping $\hat{I}$ in~\eqref{eqn:max_d_box-1}.
\qed \end{proof}

Formulation \eqref{eqn:ideal-single-relu} has a number of constraints exponential in the input dimension $\eta$, so it will not be useful directly as a MIP formulation. However, it is straightforward to separate the exponential family \eqref{eqn:ideal-single-relu-2} efficiently.

\begin{proposition} \label{prop:relu-separation}
    Take $(\hat{x},\hat{y},\hat{z}) \in [L,U] \times \bbR_{\geq 0} \times [0,1]$, along with the set
    \[
        \hat{I} = \Set{ i \in \llbracket\eta\rrbracket | w_i\hat{x}_i < w_i\left(\breve{L}(1-\hat{z}) + \breve{U}_i\hat{z}\right) }.
    \]
    If
    \[
        \hat{y} > b\hat{z} + \sum_{i \in \hat{I}} w_i\left(\hat{x}_i - \breve{L}(1-\hat{z})\right) + \sum_{i \not\in \hat{I}} w_i\breve{U}_i\hat{z},
    \]
    then the constraint in \eqref{eqn:ideal-single-relu-2} corresponding to $\hat{I}$ is the most violated in the family. Otherwise, no inequality in the family is violated at $(\hat{x},\hat{y},\hat{z})$.
\end{proposition}
\begin{proof}
Follows as a special case of Proposition~\ref{prop:max_sharp_sep}.
\qed \end{proof}

Note that \eqref{eqn:relu-big-M-2} and \eqref{eqn:relu-big-M-3} correspond to \eqref{eqn:ideal-single-relu-2} with $I = \llbracket\eta\rrbracket$ and $I = \emptyset$, respectively. All this suggests an iterative approach to formulating ReLU neurons over box domains: start with the big-$M$ formulation \eqref{eqn:relu-big-M}, and use Proposition~\ref{prop:relu-separation} to separate strengthening inequalities from \eqref{eqn:ideal-single-relu-2} as they are needed.




\subsection{The ReLU with one-hot encodings} \label{ssec:one-hot-encoding}

Although box domains are a natural choice for many applications, it is often the case that some (or all) of the first layer of a neural network will be constrained to be the product of simplices. The \emph{one-hot encoding} is a standard technique used in the machine learning community to preprocess discrete or categorical data to a format more amenable for learning (see, for example, \cite[Chapter 2.2]{Bishop:2006}). More formally, if input $x$ is constrained to take categorical values $x \in C = \{c^1,\ldots,c^t\}$, the one-hot transformation encodes this as $\tx \in \{0,1\}^{t}$, where $\tx_i = 1$ if and only if $x = c^i$. In other words, the input is constrained such that $\tx \in \underline{\Delta}^\eta \defeq \Delta^\eta \cap \{0,1\}^\eta$.

It is straightforward to construct a small ideal formulation for $\gr(\ReLu{} \circ f; \underline{\Delta}^\eta)$ as $\Set{ (x,\sum_{i=1}^\eta \max\{0, w_ix_i + b\}) | x \in \underline{\Delta}^\eta}$. However, it is typically the case that multiple features will be present in the input, meaning that the input domain would consist of the product of (potentially many) simplices. For example, neural networks have proven well-suited for predicting the propensity for a given DNA sequence to bind with a given protein~\cite{Alipanahi:2015,Zeng:2016}, where the network input consists of a sequence of $n$ base pairs, each of which can take 4 possible values. In this context, the input domain would be $\prod_{i=1}^n \underline{\Delta}^4$.

In this section, we restate the general results presented in Section~\ref{cor:simplicesD2Result}, specialized for the standard case of the ReLU nonlinearity.

\begin{corollary} \label{cor:simplicesD2ReLU}
    Presume that the input domain $D = \Delta^{p_1} \times \cdots \times \Delta^{p_\tau}$ is a product of $\tau$ simplices, and that $f(x) = \sum_{i=1}^\tau \sum_{j=1}^{p_i} w_{i,j} x_{i,j} + b$ is an affine function. Presume that, for each $i \in \llbracket \tau \rrbracket$, the weights are sorted such that $w_{i,1} \leq \cdots \leq w_{i,p_i}$. Then an ideal formulation for $\gr(\emph{\ReLu{}} \circ f; D)$ is:
    \begin{subequations}
    \begin{gather}
        y \geq w \cdot x + b \\
        \begin{aligned}
        y \leq \sum_{i=1}^\tau \left( w_{i,J(i)} z + \sum_{j=J(i)+1}^{p_i} (w_{i,j} - w_{i,J(i)})x_{i,j} \right) + bz  \\
        \forall\text{ mappings $J:\llbracket\tau\rrbracket\to\mathbb{Z}$ with $J(i)\in\llbracket p_i\rrbracket\ \forall i\in\llbracket\tau\rrbracket$} \end{aligned}\label{eqn:one-hot-2} \\
        (x,y,z) \in D \times \bbR_{\geq 0} \times \{0,1\}.
    \end{gather}
    \end{subequations}
    Moreover, a most-violated constraint from the family \eqref{eqn:one-hot-2}, if one exists, can be identified in $\mathcal{O}(p_1+\cdots+p_\tau)$ time. Finally, none of the constraints from \eqref{eqn:one-hot-2} are redundant.
\end{corollary}

\begin{proof}
Follows directly from applying Corollary~\ref{cor:simplicesD2Result} to the set $\Rsharp$.
By Corollary~\ref{cor:max_2_ideal_general}, this set actually leads to an ideal formulation, because we are taking the maximum of only two functions (with one of them being zero).  The statement about non-redundancy follows from Proposition~\ref{prop:simplicesD2FacetDefining}.
\qed\end{proof}

\subsection{The leaky ReLU over a box domain}
A slightly more exotic variant of the ReLU is the \emph{leaky ReLU}, defined as $\Leaky{}(v;\alpha) = \max\{\alpha v, v\}$ for some constant $0 < \alpha < 1$. Instead of fixing any negative input to zero, the leaky ReLU scales it by a (typically small) constant $\alpha$. This has been empirically observed to help avoid the ``vanishing gradient'' problem during the training of certain networks~\cite{Maas:2013,Xu:2015b}. We present analogous results for the leaky ReLU as for the ReLU: an ideal MIP formulation with an efficient separation routine for the constraints. 
\begin{proposition} \label{prop:ideal-leaky-relu}
    Take some affine function $f(x) = w \cdot x + b$ over input domain $D = [L,U]$. The following is a valid formulation for $\gr(\emph{\Leaky{}} \circ f; [L,U])$:
    \begin{subequations} \label{eqn:leaky-big-M}
    \begin{align}
        y &\geq f(x) \label{eqn:leaky-big-M-1} \\
        y &\geq \alpha f(x) \label{eqn:leaky-big-M-2} \\
        y &\leq f(x) - (1-\alpha) \cdot M^-(f; [L,U]) \cdot (1-z) \label{eqn:leaky-big-M-3} \\
        y &\leq \alpha f(x) - (\alpha-1) \cdot M^+(f; [L,U]) \cdot z \label{eqn:leaky-big-M-4} \\
        (x,y,z) &\in [L,U] \times \bbR \times [0,1] \\
        z &\in \{0,1\}.
    \end{align}
    \end{subequations}
    Moreover, an ideal formulation is given by \eqref{eqn:leaky-big-M}, along with the constraints
    \begin{align} \label{eqn:leaky-cut}
        y \leq& \left(\sum_{i \in I} w_i(x_i - \breve{L}_i(1-z)) + \left(b + \sum_{i \not\in I} w_i\breve{U}_i\right)z\right) \nonumber \\
        & + \alpha\left(\sum_{i \not\in I} w_i(x_i - \breve{U}_iz) + \left(b + \sum_{i \in I} w_i\breve{L}_i\right)(1-z)\right) \quad \forall I \subseteq \llbracket \eta \rrbracket.
    \end{align}
    Additionally, the most violated inequality from the family \eqref{eqn:leaky-cut} can be separated in $\mathcal{O}(\eta)$ time. Finally, each inequality in \emph{(\ref{eqn:leaky-big-M-1}--\ref{eqn:leaky-big-M-4})} and \eqref{eqn:leaky-cut} is facet-defining.
\end{proposition}
\begin{proof}
Follows as a special case of Corollary~\ref{cor:max_2_ideal}.
\qed \end{proof}
\section{Computational experiments} \label{sec:computational}
To conclude this work, we perform a preliminary computational study of our approaches for ReLU-based networks. We focus on verifying image classification networks trained on the canonical MNIST digit data set~\cite{LeCun:1998}. We train a neural network $f : [0,1]^{28 \times 28} \to \bbR^{10}$, where each of the $10$ outputs corresponds to the logits\footnote{\blue{In this context, logits are non-normalized predictions of the neural network so that $\tilde{x} \in [0,1]^{28 \times 28}$ is predicted to be digit $i-1$ with probability $\exp(f(\tilde{x})_i)/\sum_{j=1}^{10}\exp(f(\tilde{x})_j)$    \cite{logits-def}.}} for each of the digits from 0 to 9. Given a training image $\tilde{x} \in [0,1]^{28 \times 28}$, our goal is to prove that there does not exist a perturbation of $\tilde{x}$ such that the neural network $f$ produces a wildly different classification result. If $f(\tilde{x})_i = \max_{j=1}^{10} f(\tilde{x})_j$, then $\tilde{x}$ is placed in class $i$. Consider an input image with known label $i$. To evaluate robustness around $\tilde{x}$ with respect to class $j$, select some small $\epsilon > 0$ and solve the problem $\max_{a : ||a||_\infty \leq \epsilon}\: f(\tilde{x} + a)_j - f(\tilde{x} + a)_i$. If the optimal solution (or a dual bound thereof) is less than zero, this verifies that our network is robust around $\tilde{x}$ as we cannot produce a small perturbation that will flip the classification from $i$ to $j$. 
\blue{We note that in the literature there are a number of variants of the verification problem presented above, produced by  selecting a different objective function~\cite{Liu:2019,Xiao:2018} or constraint set~\cite{Engstrom:2019}. We use a model similar to that of Dvijotham et al.~\cite{Dvijotham:2018,Dvijotham:2018a}.}


\blue{We train two models, each using the same architecture, with and without L1 regularization. The architecture starts with a convolutional layer with ReLU activation functions (4 filters, kernel size of 4, stride of 2), which has 676 ReLUs, then a linear convolutional layer (with the same parameters but without ReLUs), feeding into a dense layer of 16 ReLU neurons, and then a dense linear layer with one output per digit representing the logits. Finally, we have a softmax layer that is only enabled during training time to normalize the logits and output probabilities.} \blue{Such a network is smaller than typically used for image classification tasks, but is nonetheless capable of achieving near-perfect out of sample accuracy on the MNIST data set, and of presenting us with challenging optimization problems.} We generate 100 instances for each network by randomly selecting images $\tilde{x}$ with true label $i$ from the test data, along with a random target adversarial class $j \neq i$. 

\blue{As a general rule of thumb, larger networks will be capable of achieving higher accuracy, but will lead to larger optimization problems which are more difficult to solve. However, even with a fixed network architecture, there can still be dramatic variability in the difficulty of optimizing over different parameter realizations. We refer the interested reader to Ryu et al.~\cite{Ryu:2019} for an example of this phenomena in reinforcement learning. Moreover, in the scope of this work} we make no attempts to utilize recent techniques that train the networks to be verifiable~\cite{Dvijotham:2018,Wong:2017,Wong:2018,Xiao:2018}.

For all experiments, we use the Gurobi v7.5.2 solver, running with a single thread on a machine with 128 GB of RAM and 32 CPUs at 2.30 GHz. We use a time limit of 30 minutes (1800 s) for each run. We perform our experiments using the \texttt{tf.opt} package for optimization over trained neural networks; \texttt{tf.opt} is under active development at Google, with the intention to open source the project in the future. The \emph{big-$M$ + \eqref{eqn:ideal-single-relu-2}} method is the big-$M$ formulation \eqref{eqn:relu-big-M} paired with separation\footnote{We use cut callbacks in Gurobi to inject separated inequalities into the cut loop. While this offers little control over when the separation procedure is run, it allows us to take advantage of Gurobi's sophisticated cut management implementation.} over the exponential family \eqref{eqn:ideal-single-relu-2}, and with Gurobi's cutting plane generation turned off. Similarly, the \emph{big-$M$} and the \emph{extended} methods are the big-$M$ formulation \eqref{eqn:relu-big-M} and the extended formulation \eqref{eqn:relu-ideal-extended} respectively, with default Gurobi settings. Finally, the \emph{big-$M$ + no cuts} method turns off Gurobi's cutting plane generation without separating over \eqref{eqn:ideal-single-relu-2}. \blue{As a preprocessing step, if we can infer that a neuron is linear based on its bounds (e.g. a nonnegative lower bound or nonpositive upper bound on the input affine function of a ReLU), we transform it into a linear neuron, thus ensuring Assumption~\ref{ass:amphibious}.}

\begin{table}[tb]
    \centering
    \begin{subtable}{\linewidth}
    \centering
    \subcaption{Network with standard training.}
    \label{tab:small-relu}
    \begin{tabular}{lrrrr}
        \toprule
        method & time (s) & optimality gap & win \\ \midrule
        big-$M$ + \eqref{eqn:ideal-single-relu-2} &  174.49 & 0.53\% & 81 \\ 
        big-$M$  & 1233.49 & 6.03\% &  0 \\
        big-$M$ + no cuts & 1800.00 & 125.6\% & 0 \\
        extended &  890.21 & 1.26\% &  6 \\
        \bottomrule \\
    \end{tabular}
    \end{subtable}
    \bigskip
    \begin{subtable}{\linewidth}
    \centering
    \subcaption{Network trained with L1 regularization.}
    \label{tab:small-relu-l1}
    \begin{tabular}{lrrrr}
        \toprule
        method & time (s) & optimality gap & win \\ \midrule
        big-$M$ + \eqref{eqn:ideal-single-relu-2} & \blue{9.17} & \blue{0\%} & 100 \\ 
        big-$M$  & \blue{434.72} & \blue{1.80\%} & 0 \\
        big-$M$ + no cuts & \blue{1646.45} & \blue{21.52\%} & 0 \\
        extended & \blue{1120.14} & \blue{3.00\%} & 0 \\
        \bottomrule
    \end{tabular}
    \end{subtable}
    \caption{Shifted geometric mean for time and optimality gap taken over 100 instances (shift of 10 and 1, respectively). The ``win'' column is the number of (solved) instances on which the method is the fastest.}
\end{table}

\subsection{Network with standard training} \label{ssec:standard-training}
We start with a model trained with a standard procedure, using the Adam algorithm~\cite{Kingma:2014}, running for 15 epochs with a learning rate of $10^{-3}$. The model attains 97.2\% test accuracy. We select a perturbation ball radius of $\epsilon = 0.1$. We report the results in Table~\ref{tab:small-relu} and in Figure~\ref{fig:perf-profile}. The big-$M$ + \eqref{eqn:ideal-single-relu-2} method solves $7$ times faster on average than the big-$M$ formulation. Indeed, for 79 out of 100 instances the big-$M$ method does not prove optimality after 30 minutes, and it is never the fastest choice (the ``win'' column). Moreover, the big-$M$ + no cuts times out on every instance, implying that using \emph{some} cuts is important. The extended method is roughly 5 times slower than the big-$M$ + \eqref{eqn:ideal-single-relu-2} method, but only exceeds the time limit on 19 instances, and so is substantially more reliable than the big-$M$ method for a network of this size.



\begin{figure}[tpb]
\begin{subfigure}{\linewidth}
\subcaption{Network with standard training.}
\label{fig:perf-profile}
\centering
\begin{tikzpicture}
\begin{axis}[
	xlabel=Time (s),
	ylabel=Number of instances solved,
	xmin=1,
	xmax=1800,
	xmode=log,
	ymin=0,
	ymax=100,
	cycle multi list={Set1-4},
	thick,
	legend pos=outer north east,
	x tick label style={/pgf/number format/.cd, set thousands separator={}},
	scale=0.8
    ]

\addplot +[dotted, line width=1.5] coordinates {
(0.000,0)(200.016,0)(213.400,1)(250.701,2)(255.209,3)(389.366,4)(416.042,5)(428.185,6)(434.593,7)(477.248,8)(482.728,9)(549.913,10)(579.752,11)(626.431,12)(737.075,13)(872.758,14)(915.537,15)(958.458,16)(1012.978,17)(1147.197,18)(1256.312,19)(1573.421,20)(1800.000,20)
};

\addplot +[line width=1.5] coordinates { (0,0)
(0.000,0)(10.926,0)(11.029,1)(12.141,2)(12.574,3)(13.178,4)(13.310,5)(13.519,6)(14.511,7)(15.143,8)(16.212,9)(18.118,10)(19.818,11)(22.535,12)(22.964,13)(26.257,14)(30.373,15)(30.620,16)(33.884,17)(33.910,18)(35.863,19)(36.134,20)(44.185,21)(44.710,22)(46.004,23)(49.104,24)(52.136,25)(55.366,26)(56.669,27)(56.678,28)(62.419,29)(63.682,30)(64.232,31)(70.478,32)(75.827,33)(83.694,34)(88.192,35)(89.130,36)(93.818,37)(93.927,38)(96.258,39)(102.785,40)(108.975,41)(116.348,42)(117.292,43)(130.975,44)(135.703,45)(138.493,46)(140.969,47)(161.391,48)(162.063,49)(186.414,50)(186.478,51)(202.393,52)(213.807,53)(216.762,54)(221.753,55)(236.515,56)(237.472,57)(275.902,58)(304.160,59)(311.763,60)(466.913,61)(516.754,62)(564.719,63)(567.466,64)(596.327,65)(676.548,66)(711.947,67)(736.712,68)(783.253,69)(869.017,70)(895.416,71)(917.905,72)(940.710,73)(946.208,74)(955.712,75)(957.823,76)(1008.568,77)(1031.457,78)(1179.609,79)(1315.609,80)(1327.211,81)(1508.650,82)(1556.477,83)(1676.604,84)(1800.000,84)
};

\addplot +[dashed, line width=1.5] coordinates { (0,0)
(0.000,0)(291.123,0)(304.260,1)(337.765,2)(338.746,3)(354.233,4)(370.973,5)(413.280,6)(424.993,7)(460.116,8)(461.748,9)(463.096,10)(464.743,11)(477.083,12)(482.457,13)(483.141,14)(483.467,15)(500.758,16)(512.934,17)(523.865,18)(532.305,19)(546.597,20)(551.939,21)(560.813,22)(566.648,23)(567.355,24)(572.290,25)(625.472,26)(630.244,27)(638.643,28)(646.618,29)(649.133,30)(675.245,31)(678.859,32)(680.776,33)(691.489,34)(704.245,35)(707.126,36)(722.259,37)(725.898,38)(755.967,39)(775.662,40)(781.173,41)(791.785,42)(800.285,43)(808.156,44)(812.561,45)(819.229,46)(881.397,47)(887.791,48)(892.458,49)(918.828,50)(923.877,51)(937.224,52)(941.182,53)(949.978,54)(957.818,55)(965.734,56)(1038.186,57)(1087.226,58)(1113.039,59)(1116.502,60)(1139.682,61)(1139.921,62)(1154.051,63)(1202.194,64)(1211.171,65)(1211.259,66)(1218.455,67)(1222.273,68)(1244.772,69)(1273.453,70)(1327.791,71)(1329.087,72)(1361.641,73)(1384.576,74)(1389.731,75)(1427.386,76)(1448.842,77)(1471.646,78)(1516.231,79)(1517.456,80)(1800.000,80)
};


\addplot +[dashdotted, line width=1.5] coordinates {
(0.000,0)(1800,0)
};

\legend{big-$M$, big-$M$ + \eqref{eqn:ideal-single-relu-2}, extended, big-$M$ + no cuts}
\end{axis}
\end{tikzpicture}
\end{subfigure}
\par\medskip
\begin{subfigure}{\linewidth}
\subcaption{Network trained with L1 regularization.}
\label{fig:perf-profile-l1}
\centering
\begin{tikzpicture}
\begin{axis}[
	xlabel=Time (s),
	ylabel=Number of instances solved,
	xmin=1,
	xmax=1800,
	xmode=log,
	ymin=0,
	ymax=100,
	cycle multi list={Set1-4},
	thick,
	legend pos=outer north east,
	x tick label style={/pgf/number format/.cd, set thousands separator={}},
	scale=0.8
    ]

\addplot +[dotted, line width=1.5] coordinates {
(0.000,0)(2.672,0)(4.194,1)(4.495,2)(5.500,3)(6.059,4)(6.617,5)(7.319,6)(7.679,7)(7.933,8)(8.579,9)(9.416,10)(10.214,11)(11.252,12)(12.325,13)(14.003,14)(14.139,15)(16.680,16)(17.267,17)(18.140,18)(18.302,19)(19.062,20)(22.982,21)(23.670,22)(23.851,23)(24.166,24)(25.246,25)(28.146,26)(29.314,27)(32.740,28)(33.392,29)(39.119,30)(61.101,31)(92.941,32)(105.757,33)(121.783,34)(143.929,35)(151.098,36)(153.534,37)(179.871,38)(182.272,39)(182.310,40)(196.572,41)(199.991,42)(225.512,43)(227.084,44)(243.006,45)(255.176,46)(259.267,47)(302.941,48)(313.027,49)(335.576,50)(421.178,51)(510.118,52)(529.271,53)(578.211,54)(611.782,55)(612.747,56)(725.068,57)(732.296,58)(763.151,59)(893.462,60)(934.758,61)(1041.979,62)(1102.970,63)(1184.298,64)(1208.981,65)(1229.530,66)(1254.524,67)(1526.415,68)(1665.686,69)(1800.000,69)
};

\addplot +[line width=1.5] coordinates {
(0.000,0)(0.735,0)(0.801,1)(0.833,2)(1.064,3)(1.363,4)(1.403,5)(1.408,6)(1.477,7)(1.500,8)(1.538,9)(1.624,10)(1.628,11)(1.638,12)(1.672,13)(1.702,14)(1.724,15)(1.807,16)(1.867,17)(1.940,18)(2.012,19)(2.116,20)(2.146,21)(2.209,22)(2.289,23)(2.320,24)(2.376,25)(2.415,26)(2.531,27)(2.545,28)(2.604,29)(2.645,30)(2.695,31)(2.780,32)(2.818,33)(2.836,34)(2.850,35)(3.170,36)(3.271,37)(3.398,38)(3.413,39)(3.414,40)(3.482,41)(3.498,42)(3.537,43)(3.546,44)(3.659,45)(3.671,46)(3.683,47)(3.793,48)(3.956,49)(4.143,50)(4.261,51)(4.262,52)(4.381,53)(4.810,54)(4.917,55)(5.184,56)(5.443,57)(6.458,58)(6.545,59)(6.557,60)(6.576,61)(6.645,62)(7.028,63)(7.326,64)(7.617,65)(8.397,66)(8.939,67)(9.044,68)(9.761,69)(10.163,70)(10.277,71)(11.311,72)(11.905,73)(12.156,74)(12.385,75)(12.953,76)(12.991,77)(13.334,78)(14.081,79)(14.219,80)(14.784,81)(15.786,82)(16.550,83)(16.655,84)(17.168,85)(18.588,86)(21.297,87)(22.294,88)(24.003,89)(25.354,90)(27.488,91)(37.696,92)(41.453,93)(45.450,94)(49.440,95)(54.555,96)(66.422,97)(88.877,98)(114.365,99)(1800.000,99)
};

\addplot +[dashed, line width=1.5] coordinates { (0,0)
(0.000,0)(208.316,0)(231.899,1)(331.030,2)(383.609,3)(405.200,4)(406.205,5)(438.488,6)(461.671,7)(464.766,8)(473.584,9)(489.581,10)(504.198,11)(508.922,12)(509.089,13)(527.268,14)(528.185,15)(531.285,16)(535.875,17)(537.501,18)(542.574,19)(561.391,20)(561.452,21)(568.160,22)(572.659,23)(607.407,24)(614.894,25)(618.583,26)(621.456,27)(674.113,28)(681.671,29)(721.298,30)(722.812,31)(734.094,32)(742.035,33)(772.447,34)(774.719,35)(796.008,36)(808.028,37)(827.122,38)(833.355,39)(848.235,40)(872.496,41)(882.005,42)(885.884,43)(904.149,44)(969.544,45)(969.810,46)(975.263,47)(1011.177,48)(1015.720,49)(1017.963,50)(1024.353,51)(1055.543,52)(1084.240,53)(1099.825,54)(1137.626,55)(1161.919,56)(1175.886,57)(1240.898,58)(1296.375,59)(1376.420,60)(1390.923,61)(1487.599,62)(1582.255,63)(1590.398,64)(1590.811,65)(1671.488,66)(1705.309,67)(1736.513,68)(1800.000,68)
};

\addplot +[dashdotted, line width=1.5] coordinates {
(0.000,0)(25.238,0)(58.059,1)(58.896,2)(68.126,3)(219.182,4)(378.731,5)(646.885,6)(1407.056,7)(1537.978,8)(1617.722,9)(1734.822,10)(1800.000,10)
};

\legend{big-$M$, big-$M$ + \eqref{eqn:ideal-single-relu-2}, extended, big-$M$ + no cuts}
\end{axis}
\end{tikzpicture}
\end{subfigure}
\caption{Number of instances solved within a given amount of time. Curves to the upper left are better, with more instances solved in less time.}
\end{figure}
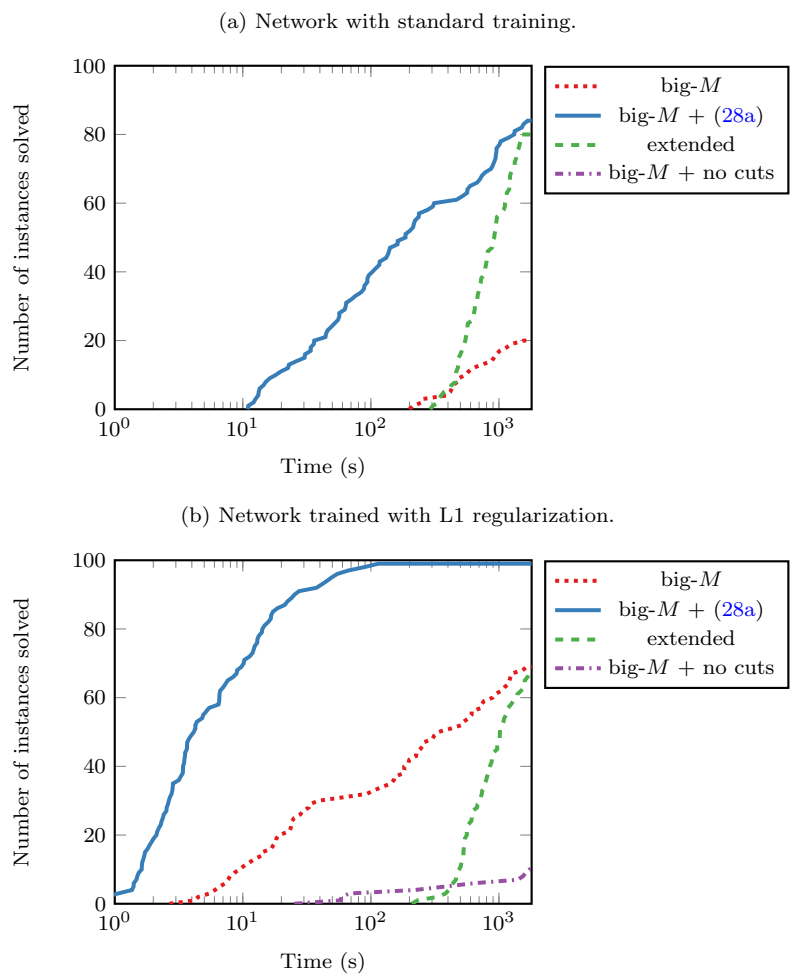

\subsection{ReLU network with L1 regularization} \label{ssec:l1-regularization}

The optimization problems studied in Section~\ref{ssec:standard-training} are surprisingly difficult given the relatively small size of the networks involved. This can largely be attributed to the fact that the weights describing the neural network are almost completely dense. To remedy this, we train a second model, using the same network architecture, but with L1 regularization as suggested by Xiao et al.~\cite{Xiao:2018}. We again set a radius of $\epsilon=0.1$, and train for 100 epochs with a learning rate of $5 \cdot 10^{-4}$ and a regularization parameter of $10^{-4}$. \blue{The network achieves a test accuracy of 98.0\%. With regularization, 4\% of the weights in the network are zero, compared to 0.3\% in the previous network. Moreover, with regularization we can infer from variable bounds that on average 73.1\% of the ReLUs are linear within the perturbation box of radius $\epsilon$, enabling us to eliminate the corresponding binary variables. In the first network, we can do so only for 27.8\% of the ReLUs.}


We report the corresponding results in Table~\ref{tab:small-relu-l1} and Figure~\ref{fig:perf-profile-l1}. While the extended approach does not seem to be substantially affected by the network sparsity, the big-$M$-based approaches are able to exploit it to solve more instances, more quickly. The big-$M$ approach is able to solve 70 of 100 instances to optimality within the time limit, though the mean solve time is still quite large. In contrast, the big-$M$ + \eqref{eqn:ideal-single-relu-2} approach fully exploits the sparsity in the model, solving each instance strictly faster than each of the other approaches. Indeed, the approach is able to solve 69 instances in less than 10 seconds, and solves all instances in under 120 seconds.

\bibliographystyle{spmpsci}
\bibliography{other.bib}

\appendix

\section{A tight big-M formulation for \Max{} on polyhedral domains}\label{app:max_big_m}


We present a tightened big-$M$ formulation for the maximum of $d$ affine functions over an arbitrary polytope input domain. We can view the formulation as a relaxation of the system in Proposition~\ref{prop:max_ideal_set_dual}, where we select $d$ inequalities from each of \eqref{eqn:max_ideal_set_dual-1} and~\eqref{eqn:max_ideal_set_dual-2}: those corresponding to $\bar{\alpha}, \ubar{\alpha} \in \{ w^1, \ldots, w^d \}$. This subset yields a valid formulation, and we obviate the need for direct separation. This formulation can also be viewed as an application of Proposition 6.2 of Vielma~\cite{Vielma:2015}, and is similar to the big-$M$ formulations for generalized disjunctive programs of Trespalacios and Grossmann~\cite{Trespalacios:2014}.

\begin{proposition}\label{prop:max_big_m}
Take coefficients $N$ such that, for each $\ell, k \in \llbracket d \rrbracket$ with $\ell \neq k$,
    \begin{subequations} \label{eqn:optimal-coefficients_app}
    \begin{align}
        N^{\ell,k,+} &\geq \max_{x^k \in D_{|k}}\{(w^k - w^\ell) \cdot x^k\} \label{eqn:optimal-coefficients-1_app}\\
        N^{\ell,k,-} &\leq \min_{x^k \in D_{|k}}\{(w^k - w^\ell) \cdot x^k\} ,\label{eqn:optimal-coefficients-2_app}
    \end{align}
    \end{subequations}
and $N^{k,k,+} = N^{k,k,-} = 0$ for all $k \in \llbracket d \rrbracket$. Then a valid formulation for $\gr(\emph{\Max{}} \circ (f^1,\ldots,f^d); D)$ is:
\begin{subequations}\label{eqn:max_big_m_app}
\begin{align}
    y \leq w^\ell \cdot x + \sum_{k=1}^d (N^{\ell,k,+} + b^k) z_k&\quad\forall \ell \in \llbracket d \rrbracket\label{eqn:max_big_m-1_app}\\
    y \geq w^\ell \cdot x + \sum_{k=1}^d (N^{\ell,k,-} + b^k) z_k&\quad\forall \ell \in \llbracket d \rrbracket\label{eqn:max_big_m-2_app}\\
    (x,y,z) \in D \times \bbR \times \Delta^d \label{eqn:max_big_m-3_app} \\
    z \in \{0,1\}^d\label{eqn:max_big_m-4_app}
\end{align}
\end{subequations}
\end{proposition}

The tightest possible coefficients in \eqref{eqn:optimal-coefficients_app} can be computed exactly by solving an LP for each pair of input affine functions $\ell \neq k$. While this might be exceedingly computationally expensive if $d$ is large, it is potentially viable if $d$ is a small fixed constant. For example, the max pooling neuron computes the maximum over a rectangular window in a larger array~\cite[Section 9.3]{Goodfellow:2016}, and is frequently used in image classification architectures. Typically, max pooling units work with a $2 \times 2$ or a $3 \times 3$ window, in which case $d=4$ or $d=9$, respectively.

In addition, if in practice we observe that if the set $D_{|k}$ is empty, then we can infer that the neuron is not irreducible as the $k$-th input function is never the maximum, and we can safely prune it. In particular, if we attempt to compute the coefficients for $z_k$ and it is proven infeasible, we can prune the $k$-th function.

\end{document}